\documentclass[a4paper,11pt]{article}

\usepackage[pdftex
]{hyperref}

\usepackage{amssymb}
\usepackage{amsfonts}
\usepackage{amsmath}
\usepackage[bbgreekl]{mathbbol}

\usepackage{graphicx}
\usepackage[all]{xy}
\usepackage{array}

\usepackage{amsthm}

\newtheorem{theorem}{Theorem}[section]
\newtheorem*{theoremnn}{Theorem}
\newtheorem{lemma}[theorem]{Lemma}

\newtheorem{corollary}[theorem]{Corollary}
\newtheorem*{corollarynn}{Corollary}

\theoremstyle{definition}
\newtheorem{definition}[theorem]{Definition}
\newtheorem{notation}[theorem]{Notation}
\newtheorem{terminology}[theorem]{Terminology}
\newtheorem{example}[theorem]{Example}

\theoremstyle{remark}
\newtheorem{remark}[theorem]{Remark}

\newcommand{\Si}{\mathfrak{S}}
\newcommand{\Ext}{{\mathrm{Ext}}}
\renewcommand{\hom}{{\mathrm{Hom}}}
\newcommand{\Id}{{\mathrm{Id}}}
\renewcommand{\k}{\mathbb{k}}

\newcommand{\pr}{\mathrm{pr}}

\newcommand{\op}{\mathrm{op}}

\newcommand{\cross}{\times}
\newcommand{\kmod}{{\text{$\k$-mod}}}
\newcommand{\kvect}{{\text{$\k$-vect}}}
\newcommand{\Gmod}{{\text{$G$-mod}}}
\renewcommand{\i}{\mathrm{in}}
\renewcommand{\P}{\mathcal{P}}
\newcommand{\V}{\mathcal{V}}
\newcommand{\A}{\mathcal{A}}

\newcommand{\C}{\mathcal{C}}

\newcommand{\F}{\mathcal{F}}

\renewcommand{\H}{H_{\mathrm{rat}}}
\newcommand{\Oplus}{{\textstyle\bigoplus}}

\title{Cohomology of classical algebraic groups from the functorial viewpoint}
\author{Antoine Touz\'e}
\date{\today}
\begin{document}

\maketitle
\sloppy
\begin{abstract}
We prove that extension groups in strict polynomial functor categories compute
the rational cohomology of classical algebraic groups. 
This result was 
previously known only for general linear groups. We give several applications 
to the study of classical algebraic groups, 
such as a cohomological stabilization property, the injectivity of
external cup products, and the existence of Hopf algebra structures on the
(stable) cohomology of a classical algebraic group with coefficients 
in a Hopf algebra.
Our result also opens the way to new explicit cohomology computations. We give
an example inspired by recent computations of Djament and Vespa.
\end{abstract}

\tableofcontents

\section{Introduction}

Over the past fifteen years, the relations between functor categories and
the cohomology of the algebraic general linear group $GL_n$ 
have been successfully used to
prove cohomological finite generation conjectures \cite{FS,TVdK}, and 
they have also proved very useful to perform explicit cohomology computations
\cite{FFSS,Chal1,FF}. 
The first purpose of this paper is to extend these relations 
to other classical algebraic groups. More specifically, we prove that if $G$ is a
symplectic group, an orthogonal group, a general linear group, or more
generally a finite product of these groups, then $\Ext$-groups in a suitable
functor category compute the cohomology of $G$. The second purpose of this
paper is to illustrate some advantages of the functorial point of view. In
particular, we obtain new cohomological results for classical algebraic
groups, whose proofs do not seem to belong to the usual algebraic group
setting.

The cohomology we treat here is the cohomology of algebraic groups of
\cite{Jantzen}, which was
introduced by Hochschild (it is often called `rational
cohomology' to emphasize that it arises from rational representations).    
The functors which play a role in the algebraic group setting are
the `strict polynomial functors' of Friedlander and Suslin 
\cite{FS}, and their multivariable analogues. 
Our results are the algebraic counterpart of
 recent results of Djament and Vespa
\cite{DjamentVespa} 
about the finite groups $O_{n,n}(\mathbb{F}_q)$, $Sp_n(\mathbb{F}_q)$.
%Our results may be interpreted as algebraic group analogues of recent results
%of Djament and Vespa \cite{DjamentVespa} about the finite groups $O_{n,n}(\mathbb{F}_q)$, $Sp_n(\mathbb{F}_q)$.
However, the methods required for algebraic groups 
are very different from those needed for finite groups.
The cohomological stabilization property illustrates this difference vividly:
in the algebraic group setting, it is an immediate consequence of the link
between extension groups in functor categories and cohomology of algebraic
groups, while in the finite group setting these two results are independent.

What follows is a synopsis of the results of the paper. 

\subsubsection*{Relating functor categories to the cohomology of
classical groups}

In section \ref{sec-3}, we
establish the link between $\Ext$-groups 
in strict polynomial functor categories and rational cohomology of general
linear, orthogonal and symplectic groups.
For example, we prove:
\begin{theoremnn}[\ref{thm-Spn}, the symplectic case]
Let $\k$ be a commutative ring, and let $n$ be a positive integer.
For any $F\in\P$
we have a $*$-graded map,
natural in $F$:
$$\phi_{Sp_n,F}:\Ext^*_{\P}(\Gamma^\star(\Lambda^2),F)\to \H^*(Sp_n,F_n)\;. $$
The map $\phi_{Sp_n,F}$ is compatible with cup products:
$$\phi_{Sp_n,F\otimes F'}(x\cup y)= \phi_{Sp_n,F}(x)\cup
\phi_{Sp_n,F'}(y)\;.$$
Moreover, $\phi_{Sp_n,F}$ is an isomorphism whenever $2n\ge \deg(F)$. 
\end{theoremnn}

\noindent
Here `$\P$' refers to the category of strict polynomial functors of
Friedlander and Suslin. So if $\V_\k$ is the category of finitely generated
projective $\k$-modules, objects of $\P$ are functors
$F:\V_\k\to\V_\k$ with an additional `polynomial structure' which
ensures that the image $F(V)$ 
of a rational $G$-module $V$ is a
rational representation of the algebraic group $G$. The rational
$Sp_n$-module $F_n$ is obtained by evaluating $F$ on the dual of the standard
representation $\k^{2n}$ of $Sp_n$. The cup product on the left comes from the
usual coalgebra structure on the divided powers $\Gamma^\star(\Lambda^2)$.

Our method is based on classical invariant theory \cite{DCP}. The proofs
for orthogonal, symplectic and general linear groups are analogous. For the
orthogonal and symplectic groups the results are new. In the
general linear case, we obtain a new treatment (and a generalization over a
commutative ring $\k$) of previously known results: \cite[Cor 3.13]{FS}, \cite[Thm
1.5]{FF} and \cite[Thm 1.3]{Touze}.

In section \ref{sec-4}, we use K\"unneth formulas to
extend these results when $G_n$ is a finite product 
of general linear, orthogonal and/or symplectic groups. In that case, one has
to consider the category $\P_G$ of 
strict polynomial functors $F$ `adapted to $G_n$', that is
with a number of variables taking into account 
the number of factors in the product $G_n$. Evaluation of $F$ on
specific representations of the factors of $G_n$ yield a rational 
$G_n$-module $F_n$
and we have:
\begin{theoremnn}[\ref{thm-generique}]
Let $\k$ be a commutative ring, let $n$ be a positive integer
and let $G_n$ be a finite product of the algebraic groups (over $\k$)
$GL_n$, $Sp_n$ and $O_{n,n}$. For any $F\in\P_G$
 we have a $*$-graded map,
natural in $F$, which is compatible with cup products:
$$\phi_{G_n,F}:\Ext^*_{\P_G}(\Gamma^\star(F_G),F)\to \H^*(G_n,F_n) $$

Assume that $2n$ is greater or equal to the degree of $F$. If one of the factors of
$G_n$ equals $O_{n,n}$, 
assume furthermore that $2$ is invertible in $\k$. 
Then $\phi_{G_n,F}$ is an isomorphism. 
\end{theoremnn}

\subsubsection*{Some applications of the functorial 
viewpoint in algebraic group cohomology}

As a first application, we deduce from theorem \ref{thm-generique} 
a cohomological stabilization property.

\begin{corollarynn}[\ref{cor-stab}]
Let $\k$ be a commutative ring, let $n$ be a positive integer
and let $G_n$ be a finite product of copies of
$GL_n$, $Sp_n$ or $O_{n,n}$. Let $F\in\P_G$ be a degree $d$ functor adapted to
$G_n$. Let $n,m$ be two positive integers such that $2m\ge 2n\ge d$. If 
the orthogonal group appears as one of the factors of $G_n$, assume
furthermore that $2$ is invertible in $\k$. Then we have an isomorphism
$$\phi_{n,m}:\H^*(G_m,F_m)\xrightarrow[]{\simeq} \H^*(G_n,F_n)\;.$$
\end{corollarynn}

\noindent
We shall denote by $\H^*(G_\infty,F_\infty)$ the stable value 
of $\H^*(G_n,F_n)$
(though this stable value is obtained for relatively small values of $n$).

As a second application we obtain a striking injectivity property for cup
products. In general, if $G$ is an algebraic group and if $c\in\H^*(G,M)$ and
$c'\in\H^*(G,N)$ are nontrivial cohomology classes, their (external)
cup product $c\cup c'\in\H^*(G,M\otimes N)$ may very well be zero. 
For example, if $\k$ is a field of odd 
characteristic and $G_a$ is the
additive group, then the cohomology algebra $\H^*(G_a,\k)$ is \cite{CPSVdK} a
free commutative graded algebra with generators $(x_i)_{i\ge 0}$ of degree $2$
and generators $(\lambda_i)_{i\ge 0}$ of degree one. Since the multiplication
$\k\otimes\k\to\k$ is an isomorphism, it is not hard to build pairs of non
trivial classes $(\alpha,\beta)$ whose external cup product $\alpha\cup\beta$
is zero. This cancellation phenomenon does not occur in (stable) cohomology of
classical groups over a field.

\begin{corollarynn}[\ref{cor-inj}]
Let $\k$ be a field. 
Let $G_n$ be a product of copies of the groups $GL_n, Sp_n$ or $O_{n,n}$, and
let $F_1,F_2$ be two functors of degree $d_1,d_2$ adapted to $G_n$. If $O_{n,n}$ is a factor in $G_n$,
assume that $\k$ has odd characteristic.
For all $n$ such that $2n\ge d_1+ d_2$, the cup product
induces a injection:
$$ \H^*(G_n, (F_1)_n)
\otimes \H^*(G_n, (F_2)_n)\hookrightarrow
\H^*(G_n, (F_1)_n\otimes (F_2)_n)\;.  $$ 
\end{corollarynn}

This results partially explains some non-vanishing phenomena, like
\cite[Lemma 4.13]{TVdK}. It follows from a more general result, namely
the existence of external coproducts in the stable
cohomology of classical groups. 

\begin{theoremnn}[\ref{thm-hopfmonstr-gp}]
Let $\k$ be a field. 
Let $G_n$ be a product of copies of the groups $GL_n, Sp_n$ or $O_{n,n}$, and
let $F_1,F_2$ be strict polynomial functors adapted to $G_n$. If $O_{n,n}$ is a factor in $G_n$,
assume that $\k$ has odd characteristic.
The stable rational
cohomology of $G_n$ is equipped with a coproduct:
$$\H^*(G_\infty, (F_1\otimes F_2)_\infty)\to \H^*(G_\infty, F_{1\;\infty})
\otimes \H^*(G_\infty, F_{2\;\infty})\;. $$
Together with the
usual cup product (cf. \S \ref{subsec-ratcohom}), 
they endow $\H^*(G_\infty,-)$ with the structure of a
graded Hopf monoidal functor (cf. definition \ref{def-Hopfmonstr}).

Moreover, the cup product is a section of the coproduct. 
\end{theoremnn}

\noindent
The construction of the external coproduct
uses the sum-diagonal adjunction, 
a feature which is specific to functor categories. Some hints that such
coproducts exist were given in \cite{FFSS}, where the authors built Hopf
algebra structures on some specific extension groups in functor categories
(when all the functors in
play are `Hopf exponential functors'). We build the external coproducts 
in section \ref{sec-prodetcoprod}, where we make a more general
attempt to classify 
the Hopf monoidal structures that may arise for extension groups in functor
categories. 

As a consequence of theorem \ref{thm-hopfmonstr-gp}, we also obtain Hopf
algebra structures (without antipode) on rational cohomology of classical groups
(compare \cite[lemma 1.11]{FFSS}):
\begin{corollarynn}[\ref{cor-Hopfalggp}]
Let $\k$ be a field. 
Let $G_n$ be a product of copies of the groups $GL_n, Sp_n$ or $O_{n,n}$, and
let $A^*$ be an $n$-graded strict polynomial functor adapted to $G_n$, endowed
with the structure of a Hopf algebra. If $O_{n,n}$ is a factor in $G_n$,
assume that $\k$ has odd characteristic. The usual cup product
$\H^*(G_\infty,A^*_\infty)^{\otimes 2}\to \H^*(G_\infty,A^*_\infty)$ may be
supplemented with a coproduct 
$\H^*(G_\infty,A^*_\infty)\to\H^*(G_\infty,A^*_\infty)^{\otimes 2}$ which
endow $\H^*(G_\infty,A^*_\infty)$ with the structure of a $(1+n)$-graded Hopf
algebra. 
\end{corollarynn}

\noindent
Such Hopf algebra structures offer a nice framework in which we can reformulate some
previously known cohomological computations, such as the existence of the universal
classes of \cite[Thm 4.1]{TVdK}, cf. corollary \ref{cor-univ}.

Finally, $\Ext$-computations in strict polynomial functor categories is a
classical subject. Many results and computational techniques are already
available. So by expressing rational cohomology of orthogonal and symplectic
groups as extension in $\P$, we open the way to new cohomology
computations. To illustrate this fact, we give one example, 
which may be proved by the method of  
Djament and Vespa \cite[\S 4.2]{DjamentVespa} 
and the computations of \cite{FFSS}:
\begin{theoremnn}[\ref{thm-calcul}]
Let $\k$ be a field of odd characteristic. Let $r$ be a nonnegative integer.
Let $S^*(I^{(r)})$ denote the symmetric algebra over the $r$-th Frobenius
twist (with $S^d(I^{(r)})$ placed in degree $2d$) and let 
$\Lambda^*(I^{(r)})$ denote the exterior powers of the $r$-th Frobenius twist
(with $\Lambda^d(I^{(r)})$ placed in degree $d$).
\begin{itemize}
\item[(i)]
The bigraded Hopf algebra 
$\H^*(O_{\infty,\infty}, S^\star(I^{(r)})_\infty)$ 
is a symmetric Hopf algebra on
generators $e_m$ of bidegree $(2m, 4)$ for $0\le m< p^r$.
\item[(ii)] 
The bigraded Hopf algebra
$\H^*(Sp_\infty, S^\star(I^{(r)})_\infty)$ is trivial.
\item[(iii)] The bigraded Hopf algebra
$\H^*(O_{\infty,\infty}, \Lambda^\star(I^{(r)})_\infty)$
is trivial.
\item[(iv)] The bigraded Hopf algebra
$\H^*(Sp_\infty, \Lambda^\star(I^{(r)})_\infty)$ is a divided power Hopf algebra on generators $e_m$ of bidegree $(2m,2)$ 
for $0\le m< p^r$. 
\end{itemize}
\end{theoremnn}

\section{Review of functor categories and group cohomology}
\label{sec-2}

\subsection{Notations}\label{subsec-nota}

If $\k$ is a commutative ring, we denote by $\V_\k$ 
the category of finitely generated projective $\k$-modules. The symbol
`$^\vee$' means $\k$-linear duality: $V^\vee:=\hom_\k(V,\k)$.

Let $V\in\V_\k$. For all $d\ge 0$, we denote by $\Gamma^d(V)$ the
$d$-th divided power of $V$, that is the invariants $(V^{\otimes d})^{\Si_d}$
where $\Si_d$ acts by permuting the factors of the tensor product (for
$d=0$, we let $\Gamma^0(V)=\k$). We also denote by $S^d(V)$, resp.
$\Lambda^d(V)$ the $d$-th
symmetric, resp. exterior, power of $V$. Let $A^*=S^*,\Lambda^*$ or
$\Gamma^*$. Then $A^*$ satisfies an `exponential isomorphism' natural in $V,W$
and associative in the obvious sense: $A^*(V\oplus W)\simeq A^*(V)\otimes
A^*(W)$. Let $\delta_2$ be the diagonal $V\to V\oplus V$, $x\mapsto (x,x)$,
and let $\Sigma_2$ be the sum $V\oplus V\to V$, $(x,y)\mapsto x+y$. 
`The' graded Hopf algebra structure on the divided
powers $\Gamma^*(V)$ (without further specification) means the following. We
consider $\Gamma^d(V)$ in degree $2d$, the unit is
$\k=\Gamma^0(V)\hookrightarrow \Gamma^*(V)$, the counit is 
$\Gamma^*(V)\twoheadrightarrow \Gamma^0(V)=\k$, the multiplication and the
comultiplication are:
$$\Gamma^*(V)^{\otimes 2}\simeq \Gamma^*(V\oplus
V)\xrightarrow[]{\Gamma^*(\Sigma_2)} \Gamma^*(V)\;,\,
\Gamma^*(V)\xrightarrow[]{\Gamma^*(\delta_2)}
\Gamma^*(V\oplus V)\simeq \Gamma^*(V)^{\otimes 2}\;.$$

\subsection{Strict polynomial
functors}\label{subsec-strpolfct}
Let $\k$ be a commutative ring and let $\A$ 
be a finite product of the categories $\V_\k$ and $\V_\k^{\op}$ (the
`$\op$' stands for the opposite category). We recall here the basic
definitions and properties of
the category of strict polynomial functors from $\A$ to $\V_\k$. The case 
$\A=\V_\k$ was introduced in \cite{FS} over a field and in \cite{SFB} over
an arbitrary commutative ring, the case $\A=\V_\k^{ \op}\times \V_\k$
corresponds to the category strict polynomial bifunctors, contravariant in the first
variable and covariant in the second one, used in \cite{FF}. The definitions
and the proofs generalize immediately when $\A$ is a more general
product.

\subsubsection*{Basic definitions}
A strict polynomial functor $F$ from $\A$ to $\V_\k$ is the following collection
of data: for each $X\in\A$, an element $F(X)\in\V_\k$ and for each $X,Y$ in
$\A$ a polynomial $F_{X,Y}\in S^*(\hom_\A(X,Y)^\vee)\otimes
\hom_\k(F(X),F(Y))$. These polynomials must satisfy two conditions: (1)
$F_{X,X}(\Id_X)=\Id_{F(X)}$, and (2) the \emph{polynomials} $(f,g)\mapsto 
F_{X,Y}(f)\circ F_{Y,Z}(g)$ and $(f,g)\mapsto F_{X,Z}(f\circ g)$ are equal.
Natural transformations between strict polynomial functors $F,G$ are linear
maps $\phi_X:F(X)\to G(X)$ such that the \emph{polynomials} $f\mapsto
G_{X,Y}(f)\circ \phi_X$ and $f\mapsto
\phi_Y\circ F_{X,Y}(f)$ are equal. Examples of strict polynomial
functors are $\hom_\A(X,-)$, 
the divided powers $\Gamma^d(\hom_\A(X,-))$ or the symmetric powers 
$S^d(\hom_\A(X,-))$. If $G$
is an affine algebraic group acting rationally on a $\k$-module $V$ and if
$F:\V_\k\to \V_\k$ is a strict polynomial functor, $F(V)$ is a rational
$G$-module ($g\in G$
acts on $F(V)$ by $v\mapsto F(g)(v)$). More generally:
\begin{lemma} Assume $\A=(\V_\k^{\op})^{\times k}\times
(\V_\k)^{\times\ell}$. Let $(G_i)_{1\le i\le k+\ell}$ be algebraic groups 
over $\k$, let
$(V_i)_{1\le i\le k}$ be right $G_i$-modules and $(V_i)_{k+1\le i\le k+\ell}$ 
be left
$G_i$-modules. Evaluation on $(V_1,\dots,V_n)$ yields a functor
from the category of strict polynomial functors with source $\A$ to the category of rational
$\textstyle\prod G_i$-modules.
\end{lemma}

A strict polynomial functor $F$ is homogeneous of degree $d$ if all the
polynomials $F_{X,Y}$ are homogeneous of degree $d$. It is of finite degree 
if the family of the degrees of the $F_{X,Y}$ is bounded. 
We denote by $\P_\A$ the category of strict polynomial functors of finite
degree with source
$\A$. Then the category $\P_\A$ splits as the direct sum of its full
subcategories $\P_{d,\A}$ of homogeneous functors of degree $d$:
$$\P_\A=\textstyle\bigoplus_{d\ge 0}\P_{d,\A}\;.$$
There is an equivalence of categories $\P_{0,\A}\simeq \V_\k$ induced by
$F\mapsto F(0,\dots,0)$. 
\begin{remark}
If $\A=(\V_\k^{\op})^{\times k}\times
(\V_\k)^{\times\ell}$, we could refine the splitting by introducing
multidegrees. Then the category $\P_{d,\A}$ would split as the
direct sum of its full subcategories of homogeneous functors of  multidegree
$(d_1,\dots,d_{k+\ell})$, with $\sum d_i =d$. For sake of simplicity, we don't
use multidegrees. Thus the term `degree' always refers to the total degree of
the functors.
\end{remark}
%If $\A=(\V_\k^{f\,\op})^{\times k}\times
%(\V_\k^{f})^{\times\ell}$, we may refine this splitting. 
%A strict polynomial functor $F$ is homogeneous of multidegree 
%$(d_1,\dots,d_{k+\ell})$ if all the polynomials $F_{X,Y}$ have multidegree
%$(d_1,\dots,d_{k+\ell})$. Then for all $d\ge 0$, 
%$\P_{d,\A}$ splits as the direct sum of its full subcategories
%$\P_{d_1,\dots,d_{k+\ell},\A}$ of homogeneous functors of multidegree 
%$(d_1,\dots,d_{k+\ell})$:
%$$\P_{d,\A}=\textstyle\bigoplus_{\sum d_i = d}\P_{d_1,\dots,d_{k+\ell},\A}\;.$$

\subsubsection*{Another presentation of strict polynomial functors}
We have defined strict polynomial functors as functors from $\A$ to $\V_\k$
endowed with an additional structure (polynomials). Equivalently, one can define
degree $d$
homogeneous strict polynomial functors as $\k$-linear 
functors from a $\k$-linear category $\Gamma^d\A$ to $\V_\k$
(cf. \cite{Pira} where T. Pirashvili credits Bousfield for this presentation).
In this presentation, the polynomial structure is encoded in the source
category $\Gamma^d\A$, 
and strict polynomial functors are genuine $\k$-linear functors,
which may make
some statements clearer.

We recall the definition of $\Gamma^d\A$. Let $d\ge 0$, and let $\A$ be a
finite product of copies of $\V_k$ or its opposite category.
The objects of $\Gamma^d\A$ are the same as the objects of $\A$, and the sets
of morphisms are the $\k$-modules
$\hom_{\Gamma^d\A}(X,Y):=\Gamma^d(\hom_\A(X,Y))$. The identity of $X$ equals
$\Id_X^{\otimes d}$. Let's define the composition. If $U,V\in\V_\k$, 
the group $\Si_d\times\Si_d$ acts by permuting the factors of the tensor
product $U^{\otimes d}\otimes V^{\otimes d}$. The diagonal inclusion
$\Si_d\simeq \Delta\Si_d\subset \Si_d\times\Si_d$ induces a morphism
$j_d:\Gamma^d(U)\otimes\Gamma^d(V)\to \Gamma^d(U\otimes V)$. The composition
in $\Gamma^d\A$ is defined as the composite:
\begin{align*}\Gamma^d(\hom_\A(X,Y))\otimes
\Gamma^d(\hom_\A(Y,Z))\xrightarrow[]{j_d} & 
\Gamma^d(\hom_\A(X,Y)\otimes \hom_\A(Y,Z))\\&\to \Gamma^d(\hom_\A(X,Z))\;,
\end{align*}
where the last map is induced by the composition in $\A$.

The following key lemma (compare \cite[Lemma 2.8 and proof of Prop. 2.9]{FS}) induces the existence of projective resolutions, and
will also have an important role in our computations.
\begin{lemma}[key lemma]\label{lm-key}
Let $d\ge 0$. Let $Y=(Y_i)\in\A$ be a tuple of free $\k$-modules,
such that each $Y_i$ has rank greater or equal to $d$. Then for all $X,Z\in\A$
the composition in
$\Gamma^d\A$ induces  an epimorphism:
$$\Gamma^d(\hom_\A(X,Y))\otimes \Gamma^d(\hom_\A(Y,Z))\twoheadrightarrow 
\Gamma^d(\hom_\A(X,Z)) $$
\end{lemma}
\begin{proof}
Using the exponential isomorphism for the divided power algebra, one reduces
to the case where $\A=\V_\k$. By naturality, one reduces furthermore to the
case where $X,Y$ are free $\k$-modules.

If $I=(d_1,\dots,d_n)$ is a tuple of positive
integers such that $\sum d_i=d$, we denote by $\Si_I$ the subgroup $\prod
\Si_{d_i}\subset \Si_d$. If $V$ is a free $\k$-module with basis $(b_i)$, and
if $b_{i_1},\dots, b_{i_n}$ are distinct elements of the basis
we let: 
$$(b_{i_1},\dots, b_{i_n},I):= \sum_{\sigma\in\Si_d/\Si_I}
\sigma.(\underbrace{b_{i_1}\otimes\dots\otimes b_{i_1}}_{\text{$d_1$
factors}}\otimes\dots\otimes \underbrace{b_{i_n}\otimes\dots\otimes b_{i_n}}
_{\text{$d_n$
factors}})\;.$$ 
Such elements form a basis of $(V^{\otimes d})^{\Si_d}$.
Now we may choose basis $(e^{Y,X}(j,i))$, $(e^{Z,Y}(k,j))$ and 
$(e^{Z,X}(k,i))$ of $\hom_\k(X,Y)$, $\hom_\k(Y,Z)$ and $\hom_\k(X,Z)$
respectively, such that $e^{Z,Y}(k,j_1)\circ e^{Y,X}(j_2,i)=e^{Z,X}(k,i)$ if
$j_1=j_2$, and $0$ in the other cases. 

To prove surjectivity, it suffices to show that for all tuple 
$I=(d_1,\dots,d_n)$ and all $n$-tuple of distinct elements 
$(e^{Z,X}(k_s,i_s))_{1\le s\le n}$, the map induced by the composition hits
$(e^{Z,X}(k_1,i_1),\dots,e^{Z,X}(k_n,i_n),I)\in 
(\hom_\k(X,Z)^{\otimes d})^{\Si_d}$.
To do this, we use that $\mathrm{rk} Y \ge d\ge n$. Thus we may choose 
\emph{distinct} 
indices $j_1,\dots,j_n$ and form the element:
$$(e^{Y,X}(j_1,i_1),\dots,e^{Y,X}(j_n,i_n),I)
\otimes (e^{Z,Y}(k_1,j_1),\dots,e^{Z,Y}(k_n,j_n),I)\;.$$
The map induced by the composition in $\Gamma^d\V_\k$ sends this element to
$(e^{Z,X}(k_1,i_1),\dots,e^{Z,X}(k_n,i_n),I)$ and we are done.
\end{proof}

\subsubsection*{Homological algebra}
Kernels, cokernels, products or sums of strict polynomial
 functors are computed in the target
category, so that categories of strict polynomial functor inherit 
the structure of $\V_\k$. Thus, if $\k$ is
a field, $\P_\A$ and $\P_{d,\A}$ are abelian categories. 
This is no longer the
case over an arbitrary commutative ring. Nonetheless, they are exact category
in the sense of Quillen \cite{Quillen}, with admissible exact sequences being the 
sequences $0\to F'\to F\to F''\to 0$
which are exact after evaluation on every object $X$.
The theory of extensions in exact categories is very similar to
the abelian one. One minor change is that $\Ext$-groups are defined in terms
of `admissible' extensions (ie: Yoneda composites of admissible short exact
sequences), so that we must use `admissible' projective or injective
resolutions to compute them (See also \cite{Buehler} for a recent
exposition).

The standard projectives are the 
functors:
$ P_X^d:=\Gamma^d(\hom_{\A}(X,-))$, for all $X\in\A$. They satisfy a 
Yoneda isomorphism, natural in $X,F$:
$$ \hom_{P_{d,\A}}(P^d_X,F)\simeq F(X)\;,\quad f\mapsto f_X(\Id_X^{\otimes
d})\;.$$
If $F$ is homogeneous of degree $d$ and if $X=(X_1,\dots,X_n)\in\A$ is a tuple of free $\k$-modules such that each
$X_i$ has a rank greater or equal to $d$, lemma \ref{lm-key} implies that the canonical map $F(X)\otimes
P_X^d\to F$ is an epimorphism. Since every epimorphism is admissible
(ie: they admit a kernel in $\P_{d,\A}$) this shows that 
$F$ has an admissible 
projective resolution by finite sums of standard projectives.

If $F\in\P_{d,\A}$, then $F^\vee:V\mapsto F(V)^\vee$ is a degree $d$
homogeneous strict
polynomial functor with source the opposite category $\A^\op$, and we have a
natural isomorphism:
$\hom_{\P_\A}(F,G^\vee)\simeq \hom_{\P_{\A^\op}}(G,F^\vee)$. By this duality,
the functors $I^d_X:=S^d(\hom_\A(X,-))=(\Gamma^d(\hom_{\A^\op}(X,-)))^\vee$
are injective. We call them `standard injectives'. They satisfy a Yoneda
isomorphism, natural in $F,X$:
$$\hom_{\P_{d,\A}}(F,I^d_X)\simeq F(X)^\vee\;,\quad f\mapsto
f_X^\vee(\Id_X^{\otimes d})\;,$$
and each $F\in\P_{d,\A}$ has an admissible injective resolution by direct sums
of standard injectives. In particular the injectives of $\P_{d,\A}$ are direct
summands of finite sums of standard injectives and we have:
\begin{lemma}\label{lm-inj-formecombinatoire}
Assume $\A=(\V_\k^{\op})^{\times k}\times
(\V_\k)^{\times\ell}$. Let $d\ge 0$.
Then for all tuple $(i_1,\dots,i_{k+\ell})$ of positive integers,
the functor 
$$I^{d}_{i_1,\dots,i_{k+\ell}}: (V_1,\dots,V_{k+\ell})\mapsto
S^d\left(\textstyle\bigoplus_{s=1}^k(V_s^\vee)^{\oplus i_s}\oplus 
\textstyle\bigoplus_{t=k+1}^{k+\ell}V_t^{\oplus i_t}\right) $$ is an injective of $\P_{d,\A}$. Moreover the
injectives of $\P_{d,\A}$ 
are direct summands of finite sums of such functors.
\end{lemma}

\subsubsection*{Examples}
We finish the presentation by giving ingredients to build examples.
First, the tensor product yields a functor $\P_{d,\A}\times\P_{d',\A}\to
\P_{d+d',\A}$. 
Let $\P_d$ be the category of degree $d$ 
homogeneous strict polynomial functors of
with source $\V_\k$. If $F\in\P_{d}$ and $G\in\P_{d',\A}$,
composition of polynomials endow $X\mapsto F(G(X))$ with the structure
of a strict polynomial functor. In that way we obtain a functor 
$\P_{d}\times\P_{d',\A}\to
\P_{dd',\A}$.
We can get numerous new examples by combining
these two methods with the following basic examples. 
The divided powers $\Gamma^d$, the symmetric powers $S^d$,
the exterior powers $\Lambda^d$ and the tensor products $\otimes^d$ are
objects of $\P_d$ (and more generally, so are the Schur functors $S_\lambda$
associated with a partition $\lambda$ of weight $d$). The natural transformations $\otimes^d\to \otimes^d$
induced by permuting the factors are morphisms in $\P_d$, as well as the
multiplication $A^{d-i}\otimes A^i\to A^d$ and the comultiplication 
$A^d\to A^{d-i}\otimes A^i$ if $A^*=S^*,\Gamma^*,\Lambda^*$. Finally, the
exponential isomorphisms $A^*(V\oplus W)\simeq A^*(V)\otimes A^*(W)$ are
morphisms of $\P_{\V_\k\times\V_\k}$.

\subsection{Functor cohomology and cup products}\label{subsec-fctcohom}

Let $E^*$ be an $n$-graded functor in $\P_\A$. We call `functor cohomology' 
the extension groups
$$\Ext^*_{\P_\A}(E^*,-)=\textstyle\bigoplus_{j,i_1,\dots,i_n}
\Ext^j_{\P_\A}(E^{i_1,\dots,i_n},-)\;.$$

If $F,G\in{\P_\A}$, we denote by $F\otimes G$ 
their tensor product $X\mapsto F(X)\otimes G(X)$. This yields a
biexact functor: 
${\P_\A}\times{\P_\A}\to {\P_\A}$.
Moreover if $F\hookrightarrow F_0\to \dots \to F_n\twoheadrightarrow E$ and 
$F'\hookrightarrow F_0'\to \dots \to F_m'\twoheadrightarrow
E'$ are two admissible extensions, their `cross product':
$$F\otimes F'\hookrightarrow  F_0\otimes F'_0\to
\dots \to (F_n\otimes E'\oplus E\otimes F_m')\twoheadrightarrow E\otimes E'$$
is once again an admissible extension (It is an exact sequence by the K\"unneth theorem,
to prove that it is admissible, one just needs to see that the kernels of its
differentials have projective values. To do this, use its exactness and that
for all $X\in\A$, $E(X)\otimes E'(X)$ is a projective $\k$-module). In this
way, we obtain an associative cross product in extension groups:
$$\cross:\Ext^*_{\P_\A}(E,F)\otimes \Ext^*_{\P_\A}(E',F')
\to \Ext^*_{\P_\A}(E\otimes E',F\otimes
F')\;.$$

Assume now that $E^*$ has an $n$-graded coalgebra structure: we have an
$n$-graded coproduct $\Delta_E:E^*\to E^*\otimes E^*$ and an augmentation
$\epsilon_E:E^*\to \k$, where $\k$ is considered as a functor of degree
$(0,\dots,0)$. 
Then we may define an external cup product 
$$\begin{array}[t]{cccc}
\cup:&\Ext^*_{\P_\A}(E^*,F)\otimes \Ext^*_{\P_\A}(E^*,F')&\to &
\Ext^*_{\P_\A}(E^*,F\otimes F')\\
& c\otimes c'&\mapsto &\Delta_E^*(c\cross c')
\end{array}\;,$$
and a unit 
$\k=\Ext^*_{\P_\A}(\k,\k)\xrightarrow[]{\epsilon_E^*}\Ext^*_{\P_\A}(E^*,\k)$, which
satisfy an associativity and a unit axiom. These axioms may be summarized by
saying that $\Ext^*_{\P_\A}(E^*,-)$ is a (multigraded) 
monoidal functor \cite[XI.2]{MLCat}.  

\subsection{Cohomology of algebraic groups and cup
products}\label{subsec-ratcohom}

Let $\k$ be a commutative ring and let $G$ be a flat algebraic group 
over $\k$ (ie: $G$
is a group scheme
represented by a $\k$-flat finitely generated Hopf algebra $\k[G]$). 
Then the category of rational $G$-modules is an abelian category with enough
injectives. The rational cohomology of $G$ with coefficients in a $G$-module
$M$ is defined as the extension groups $\H^*(G,M)=\Ext^*_{\Gmod}(\k,M)$ ($\k$ is the
trivial $G$-module).

These extension groups may be computed \cite[4.14-4.16]{Jantzen} 
as the homology of the Hochschild
complex $C^\bullet(G,M)$ with $M\otimes\k[G]^{\otimes i}$ in degree $i$.
Interpreting $C^i(G,M)$ as the set of functions $G^{\times i}\to M$, the
external cup
product 
$$\H^*(G,M)\otimes \H^*(G,N)\to \H^*(G,M\otimes N)$$
is defined at the chain level by sending $u\in C^r(G,M)$ and $v\in
C^s(G,M)$ to
$$(u\cup v)(g_1,\dots,g_{r+s}):= u(g_1,\dots,g_r)\otimes ^{g_1\dots g_r}
v(g_{r+1},\dots,g_{r+s})\;,$$
where $^gm$ denotes the image of $m\in M$ under the action of $g\in G$.
If $M=N=R$ is an algebra with a rational $G$-action, then the composite
$$C^\bullet(G,R)\otimes C^\bullet(G,R)\to C^\bullet(G,R\otimes R)
\xrightarrow[]{C^\bullet(G,m_R)}
C^\bullet(G,R)$$
is the internal cup product of \cite[Section 6.3]{TVdK}, which makes 
$\H^*(G,R)$
into a graded algebra. 

\subsubsection*{Another construction of cup products}
Now we want to give another construction of external cup products, in terms of
cross products of extensions, as we did for functor cohomology. Over a field
$\k$, this is an easy job: (i) the two constructions coincide in degree $0$,
and (ii) a $\delta$-functor argument 
\cite[XII, proof of thm 10.4]{ML} shows  
that the two constructions coincide in all degrees. Over an arbitrary ring, 
exactness of tensor products fails, so the cross product of two
extensions does not always make sense. We have a weaker statement, proved by 
\emph{ad hoc} methods.

\begin{lemma}\label{lm-2cup}Let $G$ be a flat algebraic group over a commutative
 ring $\k$ and let $M,M'$ be two $\k$-flat $G$-modules. 
 Assume that the classes $c\in H^r(G,M)$ and $c'\in H^s(G,M')$
are represented by extensions 
$M\hookrightarrow M_0\to \dots\to
M_r\to \k$ and $M'\hookrightarrow M_0'\to \dots\to
M_s'\to \k$ whose objects are $\k$-flat. Then the cross product is an exact sequence:
$$M\otimes M'\hookrightarrow M_0\otimes M'_0\to \dots\to
(M_r\otimes \k\oplus\k\otimes M_s')\twoheadrightarrow \k\otimes\k\;.$$
Its pullback by the diagonal $\Delta:\k\simeq\k\otimes\k$, $1\mapsto 1\otimes 1$
represents the external cup product $c\cup c'\in\H^{r+s}(G,M\otimes M')$.
\end{lemma}

\begin{proof}
{\bf Step 1.} 
Consider the algebra $\k[G]$ with $G$ acting by left translation. Then 
 $C^\bullet:= C^\bullet(G,\k[G])$ is a differential graded algebra 
with an action of $G$ \cite[Section 6.3]{TVdK}.
By \cite[Part I, Chap 4, sections 4.14 to 4.16]{Jantzen}, the complex $C^\bullet$ is \emph{homotopy equivalent} 
to $\k$ concentrated in degree $0$. Thus, for all $G$-modules $M,M'$, 
the multiplication of $C^\bullet$ induces a $G$-equivariant
morphism of acyclic resolutions
over $\Id_{M\otimes M'}$:
$M\otimes C^\bullet\otimes M'\otimes C^\bullet\to M\otimes M'\otimes
C^\bullet\;. $

Now $(M\otimes C^\bullet)^G=\hom_{G}(\k,M\otimes C^\bullet)$
equals the Hochschild complex
$C^\bullet(G,M)$. As a result, we have a commutative diagram:
$$\xymatrix{
\hom_G(\k,M\otimes C^\bullet)\otimes \hom_G(\k,M'\otimes
C^\bullet)\ar@{=}[r]\ar[d]^{f\otimes g\mapsto f\otimes g}
&
C^\bullet(G,M)\otimes C^\bullet(G,M')\ar[d]^-{\cup}
\\
\hom_G(\k\otimes\k,M\otimes C^\bullet\otimes M'\otimes
C^\bullet)\ar[d]^{-\circ\Delta}&
C^\bullet(G,M\otimes N)\ar@{=}[d]
\\
\hom_G(\k,M\otimes C^\bullet\otimes M'\otimes
C^\bullet)\ar[r]&
\hom_G(\k,M\otimes M'\otimes
C^\bullet)\;.
}$$
We deduce that if $c$ and $c'$ are cohomology classes represented by cycles 
$f\in\hom_G(\k,M\otimes C^\bullet)$  and 
$f'\in\hom_G(\k,M'\otimes C^\bullet)$, the cup product $c\cup c'$ is 
represented by $
(f\otimes f')\circ\Delta\in\hom_G(\k,M\otimes C^\bullet\otimes M'\otimes
C^\bullet)$.

{\bf Step 2.} Each cycle $f\in\hom_G(\k,M\otimes C^i)$ defines an
extension $E(f)$:
$M\hookrightarrow M\otimes C^0\to\dots\to M\otimes C^{i-2}\to
N^{i-1}\twoheadrightarrow \k$, where $N^{i-1}$ is the subset of all $x\in
M\otimes C^{i-1}$ such that $(\Id_M\otimes\partial)(x)$ is a multiple of
$f(1)$. 

We claim that $E(f)$ is not only exact, but also homotopy equivalent to the zero complex. 
Indeed, let $\widetilde{C}^\bullet$ denote the complex $\k\hookrightarrow C^0\to C^1\to\cdots$ (that is, $\widetilde{C}^i=C^i$ for $i\ge 0$ and $C^{-1}=\k$). Then $\widetilde{C}^\bullet$, hence $M\otimes \widetilde{C}^\bullet$, is homotopy equivalent to the zero complex.
If $s^n:M\otimes \widetilde{C}^n\to M\otimes \widetilde{C}^{n-1}$, $n\ge 0$ is the homotopy between $0$ and the identity map, then the formula: ${s^k}'=s^k$ for $k <i$ and
${s^i}'=s^i\circ f$ defines a homotopy between zero and the identity map for $E(f)$.

{\bf Step 3.} Now we turn to cross product of extensions. One easily shows
that if $E: M\hookrightarrow\dots \twoheadrightarrow\k$ and 
$E': M'\hookrightarrow\dots \twoheadrightarrow\k$ are two extensions, and if
\emph{one of the two} is either $\k$-flat or homotopy equivalent to  the zero complex, then their 
cross product $E\cross E'$ is an
exact sequence.
We derive two consequences from this: (1) $E(f)\cross E(f')$ is an extension,
and $\Delta^*(E(f)\cross E(f'))$ represents the cohomology class $[(f\otimes
f')\circ\Delta]=[f]\cup[f']$ (cf. step 1 for this equality). (2) If $E,E'$ are
$\k$-flat extensions equivalent to $E(f)$ and $E(f')$ then $\Delta^*(E\cross
E')$ is equivalent to $\Delta^*(E(f)\cross
E(f'))$. Putting (1) and (2) together, we conclude the proof. 
\end{proof}

\section{Rational cohomology of classical groups via strict
polynomial functor cohomology}\label{sec-3}

In this section, $\k$ is a commutative ring. We show that the rational  
cohomology of the general linear groups $GL_n$, the symplectic groups $Sp_n$
and the orthogonal groups $O_{n,n}$ with coefficients in functorial
representations may be computed as functor
cohomology. To be more specific, for $G=GL_n$, 
the rational cohomology is related to 
extensions in the category $\P(1,1)$ 
of functors with source $V_\k^{\op}\times \V_\k$
(ie: $\P(1,1)$ is the category of strict
polynomial bifunctors, contravariant in the first variable and covariant 
in the second
one \cite{FF}). For the orthogonal and symplectic case, the cohomology is
related to extensions in the category $\P$ of Friedlander and Suslin
\cite{FS} (ie: the category of functors with source $\V_\k$). 

Let us outline the proof. Let $G_n=Sp_n$, $O_{n,n}$ or $GL_n$. Set $\A=\V_\k$, or
$V_\k^{\op}\times \V_\k$ in the general linear case. To each $F\in\P_{\A}$, we
may associate a rational representation $F_n$ of $G_n$. 
In that way, we obtain a $\delta$-functor: $F\mapsto \H^*(G_n,F_n)$ (that is,
a nonnegatively graded functor, sending admissible short exact sequences in
$\P_\A$ to long exact sequences in $\kmod$, cf \cite{Grot}).

On the other hand, we associate to $G_n$ a `characteristic functor' 
$F_G\in\P_\A$. To be more specific, 
for $Sp_n$, resp. $O_{n,n}$, resp. $GL_n$, we take $F_G= \Lambda^2$, resp.
$S^2$, resp. $gl(-,-)=\hom_\k(-,-)$ (the characteristic functors $\Lambda^2$ and $S^2$ appear in the context of finite groups in \cite[Thm 3.21]{DjamentVespa} and $gl$ appears in \cite[Thm 1.5]{FF}).
Taking the divided powers of $F_G$, one obtains a $\delta$-functor 
$F\mapsto \Ext^*_{\P_\A}(\Gamma^\star(F_G),F)$, which is by definition
universal (ie: it vanishes on the injectives in positive $*$-degree).

Now we wish to compare these two $*$-graded $\delta$-functors (We don't
take the gradation of the divided power algebra into account) 
by the well-known elementary lemma \cite{Grot}:
\begin{lemma}\label{lm-deltafct}
Let $K^*,H^*$ be universal $\delta$-functors and let $\phi^*:K^*\to H^*$ be a
morphism of $\delta$-functors. If $\phi^0$ is an isomorphism, then for all
$i\ge 0$, $\phi^i$ is an isomorphism.
\end{lemma}
This is done in four steps.
\begin{description}
\item[Step 1:] We build a morphism of 
$\delta$-functors:
$$\phi_{G_n,-}:\Ext^*_{\P_\A}(\Gamma^\star(F_G),-)\to \H^*(G_n,-_n)\;.$$
Moreover, we check that $\phi_{G_n,-}$ is compatible with cup products. To be
more specific, the cup product on the right is the usual cup product in
rational cohomology (cf. \S \ref{subsec-ratcohom}), and the cup product on the left is
induced (cf. \S \ref{subsec-fctcohom}) by the coalgebra structure on $\Gamma^\star(F_G)$
(cf. \S \ref{subsec-nota}).
\item[Step 2:] We prove that $F\mapsto \H^*(G_n,F_n)$ is
universal. This step involves good
filtrations of $G_n$-modules. 
\item[Step 3:] We prove that the degree zero map $\phi_{G_n,F}^0$ is injective
if $2n$ is greater than the degree of $F$. This step relies on an explicit 
functor computation.
\item[Step 4:] We prove that the degree zero map $\phi_{G_n,F}^0$ is an
isomorphism if $2n$ is greater than the degree of $F$. The surjectivity is
proved via classical invariant theory.
\end{description}
Let us now give the details.

\subsection{General linear groups}

Let $\k$ be a commutative ring, and let $\P(1,1)$ be the category of 
strict 
polynomial functors with source $\V^{\op}_\k\times \V_\k$.
For any $F\in\P(1,1)$, $F_n$ denotes the
rational representation of $GL_n$ with underlying $\k$-module $F(\k^n,\k^n)$,
and with action of $g\in GL_n$ given by $F(g^{-1},g)$. In particular,
for $gl(-,-):=\hom_\k(-,-)$, one recovers
the adjoint representation $gl_n$ of $GL_n$. 
Since $\Id_{\k^n}\in gl_n$ is
invariant under the action of $GL_n$, for all $d\ge 0$ we 
have an equivariant map: 
$$\iota^d:\k\to \Gamma^d(gl_n)\,,\quad \lambda\mapsto
\lambda\Id_{\k^n}^{\otimes d}\;.$$

{\bf Step 1: construction of $\phi_{GL_n,F}$.} 
Since $F$ splits naturally as a direct
sum of homogeneous bifunctors, it suffices to do the construction for a
homogeneous bifunctor $F$. The bifunctors $\Gamma^d(gl)$ are homogeneous of
degree $2d$. As a consequence, if $F$ is homogeneous of
odd degree, then $\Ext^*_{\P(1,1)}(\Gamma^\star(gl),F)=0$ and we define 
$\phi_{GL_n,F}$ as the zero map. If  $F$ is homogeneous of
even degree $2d$, a
class $x\in \Ext^j_{\P(1,1)}(\Gamma^\star(gl),F)$ 
is represented by an admissible extension
$$0\to F\to F^0\to\dots \to F^{j-1}\to \Gamma^d(gl)\to 0\;.$$
We define $\phi_{GL_n,F}(x)\in \H^j(GL_n,F_n)=
\Ext^*_{\text{$GL_n$-mod}}(\k,F_n)$ 
as the class of the extension obtained by
evaluation on $(\k^n,\k^n)$ and pullback along $\iota^d$:
$$\iota^{d\,*}
\left(0\to F_n\to F^0_n\to\dots \to F^{j-1}_n\to \Gamma^d(gl_n)\to 0\right)\;.$$

\begin{lemma}[Completion of Step 1]\label{lm-step1-GLn}
For all  $n\ge 0$,
the map $\phi_{GL_n,-}:\Ext^*_{\P(1,1)}(\Gamma^\star(gl),-)\to
\H^*(GL_n,-_n)$ 
is a map of $\delta$-functors.
Moreover it is compatible with cup products: 
$\phi_{GL_n,F\otimes F'}(x\cup y)= \phi_{GL_n,F}(x)\cup
\phi_{GL_n,F'}(y)$.
\end{lemma}
\begin{proof}
Straightforward, except for the compatibility with cup
products, which we now give in detail. Since a bifunctor splits naturally as a direct
sum of homogeneous bifunctors, it suffices to prove the compatibility for
homogeneous $F,F'$. Furthermore, one easily reduces to the case
where
$F$ and $F'$ have even degrees $2d$ and $2d'$.
Let $E$ and $E'$ be two admissible exact sequences representing classes 
$x\in\Ext^i_{\P(1,1)}(\Gamma^d(gl),F)$
and $y\in \Ext^j_{\P(1,1)}(\Gamma^{d'}(gl),F')$. Since $E$ and $E'$ are
admissible, their kernels are bifunctors with projective values. As a result,
evaluation on
$(\k^n,\k^n)$ and pullback by $\iota^d,\iota^{d'}$ yield \emph{$\k$-projective}
 extensions
$\iota^{d\,*}(E_n), \iota^{d'\,*}(E'_n)$. By lemma \ref{lm-2cup}, the
cohomology class $\phi_{GL_n,F}(x)\cup
\phi_{GL_n,F'}(y)$ is represented by the pullback of the cross product 
$\iota^{d\,*}(E_n)\cross\iota^{d'\,*}(E'_n)$ by the diagonal
$\k\to\k\otimes\k$.
Now the diagonals 
of $\k$ and $\Gamma^*(gl)$
induce a commutative diagram
$$\xymatrix{
\k\ar[r]^-{\Delta}\ar[d]^{\iota^{d+d'}} &\k\otimes\k
\ar[d]^{\iota^{d}\otimes\iota^{d'}}\\
\Gamma^{d+d'}(gl_n)\ar[r]^-{\Delta}&\Gamma^{d}(gl_n)\otimes
\Gamma^{d'}(gl_n)\;.
}$$
Thus $\iota^{d+d' *}((E\cup E')_n)$ equals 
$\iota^{d\,*}(E_n)\cup\iota^{d'\,*}(E'_n) $ and we are done.
\end{proof}

{\bf Step 2: $F\mapsto \H^*(GL_n, F_n)$ is a universal $\delta$-functor.} Now we prove that $\H^{>0}(GL_n, F_n)$ vanishes when $F$ is an injective functor of $\P(1,1)$. 

A Chevalley group scheme over $\mathbb{Z}$ is a connected split reductive algebraic $\mathbb{Z}$-group.  A Chevalley group scheme $G$ over a commutative ring $\k$ is a group scheme obtained by base change from a Chevalley group scheme $G_\mathbb{Z}$ over $\mathbb{Z}$: $G=(G_\mathbb{Z})_\k$. If we deal with Chevalley group schemes (such as $GL_n$, $SO_{n,n}$, $Sp_n$, etc.), cohomological vanishing over arbitrary ground rings $\k$ can often be reduced 
to the case where $\k$ is a field by the following standard lemma.
\begin{lemma}\label{lm-vanish-chevalley}
Let $G_\mathbb{Z}$ be a Chevalley group scheme over the integers, acting rationally on a free $\mathbb{Z}$-module $M$ of finite type. Denote by $G_\k$ the group obtained from $G_\mathbb{Z}$ by base change. The following assertions are equivalent:
\begin{enumerate}
\item[(i)] The cohomology groups $\H^i(G_\mathbb{Z},M)$ are trivial for $i>0$.
\item[(ii)] For all field $\k$, $\H^i(G_\mathbb{k},M\otimes\k)=0$ for $i>0$.
\item[(iii)] For all commutative ring $\k$, $\H^i(G_\mathbb{k},M\otimes\k)=0$ for $i>0$.
\end{enumerate}
\end{lemma}
\begin{proof}
$(iii)\Rightarrow (i)$ is trivial. $(i)\Rightarrow (iii)$ and $(i)\Rightarrow (ii)$ follow from the universal coefficient theorem \cite[Part I, Chap 4,  Prop 4.18]{Jantzen}. So it remains to prove $(ii)\Rightarrow (i)$. By the universal coefficient theorem, $(ii)$ implies that for all field $\k$, $\H^i(G_\mathbb{Z},M)\otimes\k=0$. But $G_{\mathbb{Z}}$ is a Chevalley group scheme, so the cohomology groups $\H^i(G_\mathbb{Z},M)$ are finitely generated by \cite[Part II, Lemma B.5]{Jantzen}. So the equality $\H^i(G_\mathbb{Z},M)\otimes\k=0$ for all field $\k$ implies that $\H^i(G_\mathbb{Z},M)=0$.
\end{proof}

\begin{lemma}\label{lm-step2-GLn} Let $\k$ be a commutative ring. 
Let $J$ be an injective in the category $\P(1,1)$ of bifunctors defined over
$\k$. Then $\H^i(GL_n, J_n)=0$ if $i>0$. As a
result, $F\mapsto \H^*(GL_n, F_n)$ is a universal $\delta$-functor.
\end{lemma}
\begin{proof}
By lemma \ref{lm-inj-formecombinatoire}, it suffices to prove the vanishing on
the injectives of the form
$I^d_{k,\ell}:(V,W)\mapsto S^d((V^{\vee})^{\oplus k}\oplus {W}^{\oplus
\ell})
$, 
for $k,\ell,d\ge 0$.
The $GL_n$-module associated to $I^d_{k,\ell}$ by evaluation on 
$(\k^n,\k^n)$ is a direct summand
of the polynomial algebra over the sum $(\k^n)^{\oplus k}\oplus 
({\k^n}^\vee)^{\oplus \ell}$. 
Thus, it suffices to prove that for all integer
$k,\ell$, and for all commutative ring $\k$, we have $\H^i(GL_n,S^*(({\k^n}^\vee)^{\oplus k}\oplus 
(\k^n)^{\oplus \ell}))=0$ for $i>0$. 

By lemma \ref{lm-vanish-chevalley}, this statement reduces to the case
where $\k$ is a field. In this latter case, 
$S^*(({\k^n}^\vee)^{\oplus k}\oplus 
(\k^n)^{\oplus \ell})$ has a good filtration \cite[Section 4.9 p. 508]{AJ}. In
particular, the cohomology vanishes in positive degree.
\end{proof}

{\bf Step 3: injectivity in degree $0$.} 

\begin{lemma}\label{lm-key-GL}Let $d\ge 0$, let $n\ge d$ and let $X=\k^n$. 
There is an epimorphism: 
$$\theta:P^{2d}_{(X,X)}\twoheadrightarrow \Gamma^d(gl).$$ 
Moreover, if we evaluate the bifunctors on $(X,X)$, then $\theta_{(X,X)}$ sends
${\Id_{(X,X)}}^{\otimes 2d}\in P^{2d}_{(X,X)}(X,X)$ to ${\Id_X}^{\otimes d}\in \Gamma^d(\hom(X,X))$. 
\end{lemma}
\begin{proof} The exponential isomorphism for the divided powers induce a
epimorphism of $P^{2d}_{(X,X)}$ onto $\Gamma^d(\hom(-,X))\otimes
\Gamma^d(\hom(X,-))$. Moreover, if we evaluate on $(X,X)$, this epimorphism
sends $\Id_{(X,X)}^{\otimes 2d}$ to $\Id_{(X,X)}^{\otimes d}\otimes 
\Id_{(X,X)}^{\otimes d}$. If we postcompose this map by the map from $\Gamma^d(\hom(-,X))\otimes
\Gamma^d(\hom(X,-))$ to $\Gamma^d(gl)$ induced by composition in
$\Gamma^d\V_\k$, then the resulting map sends $\Id_{(X,X)}^{\otimes 2d}$ to
${\Id_X}^{\otimes d}$, and is an epimorphism by lemma \ref{lm-key}.
\end{proof}
  
\begin{lemma}[Completion of Step 3]\label{lm-step3-GLn}
Let $F\in\P(1,1)$ be a bifunctor defined over a commutative ring $\k$. If $2n$
is greater than the total degree of $F$, then 
$\phi^0_{GL_n,F}:\hom_{\P(1,1)}(\Gamma^*(gl), F)\to \H^0(GL_n, F_n)$
is injective.
\end{lemma} 
\begin{proof}
Since $F$ splits as a direct sum of homogeneous functors, we can restrict to
the case of homogeneous functors. Moreover, if $F$ is homogeneous of odd
degree, then $\hom_{\P(1,1)}(\Gamma^*(gl), F)=0$ and $\phi^0_{GL_n,F}$ is
injective. Now we assume that $F$ is homogeneous of degree $2d$. Let $X=\k^n$,
with $n\ge d$. By lemma
\ref{lm-key-GL}, we have a commutative diagram:
$$\xymatrix{
\hom_{\P(1,1)}(P^{2d}_{(X,X)},F)\ar[rr]^-{\simeq}  && F(X,X)\;.\\
\hom_{\P(1,1)}(\Gamma^d(gl),F)\ar[u]^-{\hom_{\P(1,1)}(\theta,F)}
\ar[rr]^-{\phi^0_{GL_n,F}} && \H^0(GL_n, F_n)\ar@{^{(}->}[u]
}$$
The horizontal arrow is the Yoneda isomorphism. Since $\theta$ is an
epimorphism, $\hom(\theta,F)$ is injective. Thus, $\phi^0_{GL_n,F}$ is
injective.
\end{proof} 

{\bf Step 4: isomorphism in degree $0$.} 
Recall from lemma \ref{lm-inj-formecombinatoire} that for all $k,\ell\ge 0$,
$I^d_{k,\ell}$ denotes the $d$-th symmetric power of the bifunctor
$(V,W)\mapsto (V^\vee)^{\oplus k}\oplus W^{\oplus \ell}$. The evaluation of
$I^d_{k,\ell}$ on the pair $(\k^n,\k^n)$ equals the $GL_n$-module of
homogeneous polynomials of
degree $d$ on the vector space $(\k^n)^{\oplus k}\oplus 
({\k^n}^\vee)^{\oplus \ell} $.
For $1\le i\le k$ and $1\le j\le \ell$ we denote by $(i|j)$ the contraction:
$$\begin{array}[t]{lccc}
(i|j): & {(\k^n)}^{\oplus k}\oplus ({\k^n}^\vee)^{\oplus \ell}&\to & \k\\
& (v_1,\dots,v_k,f_1,\dots,f_\ell)&\mapsto & f_j(v_i)
\end{array}.
$$
The contractions are homogeneous polynomials of degree two (invariant under the
action of $GL_n$), hence elements of 
$(I^2_{k,\ell})_n$. 

In fact, by \cite[Theorem 3.1]{DCP}, these contractions generate the $GL_n$-invariant subalgebra of the algebra of polynomials over ${(\k^n)}^{\oplus k}\oplus ({\k^n}^\vee)^{\oplus \ell}$. We use this fact to prove surjectivity of the $\phi^0_{GL_n,F}$ below.
\begin{lemma}\label{lm-generateurs-GLn}For all $n\ge 1$ and all $k,\ell\ge 1$, 
the contractions lie in the image of
$\phi_{GL_n,I^2_{k,\ell}}^0$.
\end{lemma}
\begin{proof}
Let $\rho:gl(V,W)\simeq V^\vee\otimes W\hookrightarrow S^2(V^\vee\oplus W)$ be
the map induced by the exponential isomorphism for $S^2$. Let $(e_i)_{1\le
i\le n}$ be a basis of $\k^n$ and let $(e_i^\vee)_{1\le
i\le n}$ be the dual basis. Then for $V=W=\k^n$, $\rho$ sends $\Id_{\k^n}=\sum
e_i^\vee\otimes e_i$ to $\sum (e_i^\vee,0)(0,e_i)$ (we denote the elements of
${\k^n}^\vee\oplus {\k^n}$ as pairs). This latter polynomial is nothing
but the polynomial $\k^n\oplus {\k^n}^\vee\to \k$, $(v,f)\mapsto f(v)$.

Now denote by $\iota_{i,j}$ the inclusion of $V^\vee\oplus W$ in the $i$-th
and $j$-th term of $(V^\vee)^{\oplus k}\oplus W^{\oplus \ell}$. Then for all
$i,j$, $\phi_{GL_n,I^2_{k,\ell}}^0$ sends $S^2(\iota_{i,j})\circ \rho$ to
$(i|j)$.
\end{proof}

\begin{lemma}\label{lm-surjectivite-GLn}
For all $k,\ell,n\ge 1$ and all $d\ge 0$, $\phi^0_{GL_n,-}$ induces an epimorphism:
$$\hom_{\P(1,1)}(\Gamma^*(gl), I^d_{k,\ell})\twoheadrightarrow
\H^0(GL_n,(I^d_{k,\ell})_n)\;.$$
\end{lemma}
\begin{proof}
By lemma \ref{lm-step1-GLn}, $\phi_{GL_n,-}$ is compatible with external cup products. In
particular, if $A^*$ is a graded bifunctor endowed with an 
algebra structure, we obtain an algebra morphism:
$$\phi_{GL_n,A^*}^0:\hom_{\P(1,1)}(\Gamma(gl),A^*)\to \H^0(GL_n,A^*_n)\;.$$
We apply this to $A^*=I^*_{k,\ell}$. By invariant theory \cite[Theorem 3.1]{DCP}, 
$\H^0(GL_n,(I^*_{k,\ell})_n)$ is generated by the contractions
$(i|j)$. By lemma \ref{lm-generateurs-GLn}, the contractions are in the image of 
$\phi_{GL_n,I^*_{k,\ell}}^0$. This proves surjectivity.
\end{proof}

\begin{lemma}[Completion of Step 4]\label{lm-step4-GLn}
Let $F\in\P(1,1)$ and let $n$ be an integer such that $2n\ge \deg F$. Then 
$\phi^0_{GL_n,F}$ is an isomorphism.
\end{lemma}
\begin{proof}
By lemma \ref{lm-inj-formecombinatoire} and by left exactness of $F\mapsto
\hom(\Gamma^*(gl), F)$ and $F\mapsto
\H^0(GL_n, F_n)$, it suffices to prove the statement for the 
$I^d_{k,\ell}$, $k,\ell\ge 1$, $d\ge 0$. For these bifunctors, the isomorphism
follows from lemmas \ref{lm-step3-GLn} and \ref{lm-surjectivite-GLn}.
\end{proof}

\begin{theorem}[The $GL_n$ case]\label{thm-GLn}
Let $\k$ be a commutative ring, and let $n$ be a positive integer.
For all $F\in\P(1,1)$
we have a $*$-graded map,
natural in $F$:
$$\phi_{GL_n,F}:\Ext^*_{\P(1,1)}(\Gamma^\star(gl),F)\to \H^*(GL_n,F_n) $$
The map $\phi_{GL_n,F}$ is compatible with cup products:
$$\phi_{GL_n,F\otimes F'}(x\cup y)= \phi_{GL_n,F}(x)\cup
\phi_{GL_n,F'}(y)\;.$$
Moreover, $\phi_{G_n,F}$ is an isomorphism whenever $2n\ge\deg(F)$. 
\end{theorem}
\begin{proof}
The first part of the theorem is given by lemma \ref{lm-step1-GLn}. It remains
to prove the isomorphism. By homogeneity, it suffices to prove the isomorphism
for homogeneous functors of degree $d\le 2n$. To do this, we restrict 
$\phi_{GL_n,-}$ to the subcategory $\P_d(1,1)$ of homogeneous 
functors of degree $d$ and
we apply lemma \ref{lm-deltafct}.
\end{proof}

\begin{remark}
This theorem was already known over a positive characteristic field $\k$: a $\k$-linear isomorphism is
built in \cite[Thm 1.5]{FF}, and compatibility with cup
products is proved in \cite[Thm 1.3]{Touze}. However,
our proof
is new and extends the result to arbitrary commutative rings. 
\end{remark}

\subsection{Symplectic groups}

Let $\k$ be a commutative ring, and let $\P$ be the category of 
strict polynomial functors with source $\V_\k$. Let $(e_i)_{1\le i\le 2n}$
be a basis of $\k^{2n}$ and let $(e_i)^\vee_{1\le i\le 2n}$ be its dual basis.
For all $n>0$ we denote by
$Sp_n$ the symplectic group, that is, the algebraic group of $2n\times 2n$
matrices preserving the skew-symmetric form: $\omega_n:=\sum_{i=1}^n
e_i^\vee\wedge e_{n+i}^\vee$. 
The standard representation
of $Sp_n$ is $\k^{2n}$ 
with left action given by matrix multiplication. For all
functor $F\in\P$, we denote by $F_n$ the rational $Sp_n$-module obtained by
evaluating $F$ on the dual $(\k^{2n})^\vee$ of the standard representation. 
In particular for $F=\Lambda^2$, $\Lambda^2_n$ is the $\k$-module of
skew-symmetric forms of degree 2. Since $\omega_n\in \Lambda^2_n$ is invariant
under the action of $Sp_n$, we have for all $d\ge 0$ an equivariant map:
$$\iota^d:\k\to \Gamma^d(\Lambda^2_n)\,,\quad \lambda\mapsto
\lambda\omega_n^{\otimes d}\;.$$

{\bf Step 1: construction of $\phi_{Sp_n,F}$.}  
By homogeneity, it suffices to do
the construction for a
homogeneous functor $F$ of degree $2d$.
In that case, a
class $x\in \Ext^j_{\P(1,1)}(\Gamma^*(\Lambda^2),F)$ 
is represented by an admissible extension
$$0\to F\to F^0\to\dots \to F^{j-1}\to \Gamma^d(\Lambda^2)\to 0\;.$$
We define $\phi_{Sp_n,F}(x)\in \H^j(Sp_n,F_n)=
\Ext^*_{\text{$Sp_n$-mod}}(\k,F_n)$ 
as the class of the extension obtained by first
evaluating on $(\k^{2n})^\vee$, and then taking the pullback along $\iota^d$.
The proof of the following lemma is analogous to the $GL_n$ case.

\begin{lemma}[Completion of Step 1]
For all $n\ge 0$,
the map $\phi_{Sp_n,-}:\Ext^\star_{\P}(\Gamma^*(\Lambda^2),-)\to
\H^\star(Sp_n,-_n)$ 
is a map of $\delta$-functors.
Moreover it is compatible with cup products: 
$\phi_{Sp_n,F\otimes F'}(x\cup y)= \phi_{Sp_n,F}(x)\cup
\phi_{Sp_n,F'}(y)$.
\end{lemma}

\begin{lemma}[Step 2] Let $\k$ be a commutative ring. 
Let $J$ be an injective in the category $\P$ of functors defined over
$\k$. Then $\H^i(Sp_n, J_n)=0$ if $i>0$. As a
result, $F\mapsto \H^*(Sp_n, F_n)$ is a universal $\delta$-functor.
\end{lemma}
\begin{proof}
By lemma \ref{lm-inj-formecombinatoire}, it suffices to prove the vanishing on
the injectives 
$I^d_{k}:V\mapsto S^d(V^{\oplus k})$, 
for $k,d\ge 0$.
As in the case of $GL_n$, it suffices to show the vanishing 
of $\H^i(Sp_n,S^*(({\k^{2n}}^\vee)^{\oplus k}))$, $i>0$, when $\k$ is a field.
Once again, this vanishing comes from the existence of a good filtration
\cite[Section 4.9 p. 508-509]{AJ}.
\end{proof}

{\bf Step 3: injectivity in degree $0$.} 
We need a variant of lemma \ref{lm-key-GL}.
\begin{lemma}\label{lm-key-Spn}
Let $d\ge 0$, let $n\ge d$ and let $X,X'$ be two copies of $\k^n$ with
respective basis $(e_i)_{1\le i\le n}$ and $(e_i)_{n+1\le i\le 2n}$. There is
an epimorphism 
$$\widetilde{\theta}:P^{2d}_{X\oplus X'}\twoheadrightarrow
\Gamma^d(\otimes^2)\;.$$
Moreover, if we evaluate the functors on $X\oplus X'$,
then $\widetilde{\theta}_{X\oplus X'}$ sends $\Id_{X\oplus X'}^{\otimes 2d}$
 to
$(\sum_{i=1}^n e_i \otimes e_{n+i})^{\otimes d}$.
\end{lemma}
\begin{proof}
The exponential formula for the divided powers induce an epimorphism from
$P^{2d}_{X\oplus X'}$ onto $\Gamma^d(\hom_\k(X,-))\otimes
\Gamma^d(\hom_\k(X',-))$. If we evaluate the functors on $X\oplus X'$, this
epimorphism sends $\Id_{X\oplus X'}^{\otimes 2d}$ to $\i_X^{\otimes d}\otimes
\i_{X'}^{\otimes d}$, where $\i_X,\i_{X'}$ are the inclusions of $X,X'$
into $X\oplus X'$. Now there is an isomorphism $X\to (X')^\vee$ which sends $e_i$ to
$e_{i+n}^\vee$ for all $i$, where $(e_{i+n}^\vee)$ is the dual basis. This
induces an isomorphism from $\Gamma^d(\hom_\k(X,-))\otimes
\Gamma^d(\hom_\k(X',-))$ to $\Gamma^d(\hom_\k(-^\vee,X'))\otimes
\Gamma^d(\hom_\k(X',-))$, which sends (after evaluation on $X\oplus X'$) 
$\i_X^{\otimes d}\otimes
\i_{X'}^{\otimes d}$ to $(\sum e_i\otimes e_{i+n})^{\otimes d}\otimes (\sum
e_{i+n}^\vee\otimes  e_{i+n})^{\otimes d}$. If we postcompose by the
map from $\Gamma^d(\hom_\k(-^\vee,X'))\otimes
\Gamma^d(\hom_\k(X',-))$ to $\Gamma^d(\otimes^2)$ induced by the composition
in $\Gamma^d\V_\k$ then by lemma \ref{lm-key} we obtain the required
epimorphism.
\end{proof}

\begin{lemma}[Completion of Step 3]\label{lm-step3-Spn}
Let $F\in\P$ be a functor defined over a commutative ring $\k$. If $2n\ge \deg
F$, then 
$\phi^0_{GL_n,F}$
is injective.
\end{lemma}
\begin{proof}
Using lemma \ref{lm-key-Spn}, we obtain an epimorphism
$\widetilde{\theta}:P^{2d}_{(\k^{2n})^\vee}\twoheadrightarrow \Gamma^d(\Lambda^2)$
which sends $\Id_{(\k^{2n})^\vee}^{\otimes 2d}$ to $ \omega_n^{\otimes d}$.
Thus, $\phi^0_{GL_n,F}$ factorizes as the composite of the injection 
$\hom_\P(\widetilde{
\theta},F)$ and the Yoneda isomorphism $\hom_\P(P^{2d}_{X\oplus X'},F)\simeq
F(X\oplus X')$. Hence $\phi^0_{GL_n,F}$ is injective.
\end{proof}

\begin{lemma}[Step 4] Let $F\in\P$ be a functor defined over a commutative ring $\k$. If $2n\ge \deg
F$, then 
$\phi^0_{Sp_n,F}$
is an isomorphism.
\end{lemma}
\begin{proof}
By lemma
\ref{lm-inj-formecombinatoire}, and left exactness of 
$F\mapsto \hom_{\P}(\Gamma^*(\Lambda),F)$ and $F\mapsto \H^0(Sp_n,F_n)$, it
suffices to prove the isomorphism for the functors of the form
$I^d_{k}:V\mapsto S^d(V^{\oplus k})$, 
for $k\ge 1$ and $d\le 2n$. Lemma \ref{lm-step3-Spn} already gives 
injectivity. It remains to prove surjectivity. But
$\phi^0_{Sp_n,I^*_k}:\hom(\Gamma^*(\Lambda^2), I^*_k)\to H^0(Sp_n, (I^*_k)_n)$ is
an algebra morphism, so we only have to prove that the generators of
$H^0(Sp_n, (I^*_k)_n)$ lie in the image of $\phi^0_{I^*_k}$. Now
$(I^*_k)_n$ is the polynomial algebra over $k$ copies of the standard
representation of $Sp_n$. Invariant
theory gives \cite[Thm 6.6]{DCP} the generators of $H^0(Sp_n, (I^*_k)_n)$: they are
homogeneous polynomials of degree two $(i|j):(\k^{2n})^{\oplus k}\to \k$, $1\le i<j\le k$, 
sending $(v_1,\dots,v_n)$ to $\omega_n(v_i,v_j)$. In particular, if $k=1$
$H^0(Sp_n, (I^*_k)_n)=\k$ and the surjectivity of $\phi^0_{Sp_n I^2_1}$ is clear.
So the proof will be completed
if we show that the $(i|j)$ lie in the image of $\phi^0_{Sp_n I^2_k}$, for $k\ge 2$.

Let $V,V'$ be two copies of $V\in\V_\k$. The exponential isomorphism for
$S^2$ yields a monomorphism $V\otimes V'\hookrightarrow S^2(V\oplus V')$. Now
if we take $V'=V$, and if we precompose by the inclusion
$\Lambda^2(V)\to V^{\otimes 2}$, we get a natural transformation
$\rho:\Lambda^2(V)\to S^2(V\oplus V)$. If $V={\k^{2n}}^\vee$, 
with basis $(e_i^\vee)_{1\le i\le 2n}$, then $\rho$ sends $e^\vee_i\wedge
e^\vee_j$ to $(e_i^\vee,0)(0,e_j^\vee)-(e_j^\vee,0)(0,e_i^\vee)$ (we denote
the elements of ${\k^{2n}}^\vee\oplus {\k^{2n}}^\vee$ as pairs). Thus, $\rho$
sends $\omega_n$ to the sum $\sum_{i=1}^n 
(e_i^\vee,0)(0,e_{i+n}^\vee)- \sum_{i=1}^n(e_{i+n}^\vee,0)(0,e_i^\vee)$, which
is nothing but the polynomial $\k^{2n}\oplus \k^{2n}\to \k$,
$(x,y)\mapsto \omega_n(x,y)$. For $i<j$, we
denote by $\iota_{i,j}$ the inclusion of
$V\oplus V$ into the $i$-th and the $j$-th term of the sum $V^{\oplus k}$.
Then $\phi_{Sp_n}$ send the natural transformation $S^2(\iota_{i,j})\circ
\rho$ to $(i|j)$ and we are done. 
\end{proof}

\begin{theorem}[The $Sp_n$ case]\label{thm-Spn}
Let $\k$ be a commutative ring, and let $n$ be a positive integer.
For all $F\in\P$
we have a $*$-graded map,
natural in $F$:
$$\phi_{Sp_n,F}:\Ext^*_{\P}(\Gamma^\star(\Lambda^2),F)\to \H^*(Sp_n,F_n) $$
The map $\phi_{Sp_n,F}$ is compatible with cup products:
$$\phi_{Sp_n,F\otimes F'}(x\cup y)= \phi_{Sp_n,F}(x)\cup
\phi_{Sp_n,F'}(y)\;.$$
Moreover, $\phi_{Sp_n,F}$ is an isomorphism whenever $2n\ge \deg(F)$. 
\end{theorem}

\subsection{Orthogonal groups}
Let $\k$ be a commutative ring, and let $\P$ be the category of 
strict polynomial functors with source $\V_\k$. Let $(e_i)_{1\le i\le 2n}$
be a basis of $\k^{2n}$ and let $(e_i)^\vee_{1\le i\le 2n}$ be its dual basis.
For all $n>0$ we denote by
$O_{n,n}$ the algebraic group of $2n\times 2n$
matrices preserving the quadratic form $q_n:=\sum_{i=1}^n e^\vee_i
e^\vee_{n+i}$. 
The standard representation
of $O_{n,n}$ is $\k^{2n}$ with left action given by matrix multiplication. For all
functor $F\in\P$, we denote by $F_n$ the rational $O_{n,n}$-module obtained by
evaluating $F$ on the dual $(\k^{2n})^\vee$ of the standard representation. 
In particular for $F=S^2$, $S^2_n$ is the $\k$-module of
polynomials of degree 2 over $\k^{2n}$. Since $q_n\in S^2_n$ is invariant
under the action of $O_{n,n}$, we have for all $d\ge 0$ an equivariant map:
$$\iota^d:\k\to \Gamma^d(S^2_n)\,\,\quad \lambda\mapsto
\lambda q_n^{\otimes d}\;.$$

The case of the orthogonal group is analogous to the case of the symplectic group, except for a restriction on the characteristic of the commutative ring $\k$ which is needed in step $2$ only.

\medskip

{\bf Step 1: construction of $\phi_{O_{n,n},F}$.} We follow rigorously
 the symplectic
case. If $F$ is homogeneous of degree $2d$,  a class
$x\in\Ext^j(\Gamma^*(S^2),F)$ is
 represented by an extension
$0\to F\to\dots\Gamma^d(S^2)\to 0$. We define $\phi_{O_{n,n},F}(x)$ as
 the class of
the extension obtained by evaluation on $(\k^{2n})^\vee$ and pullback along
$\iota^d$. We have:
\begin{lemma}[Completion of Step 1]
For all  $n\ge 0$,
the map $\phi_{O_{n,n},-}:\Ext^*_{\P}(\Gamma^\star(S^2),-)\to
\H^*(O_{n,n},-_n)$ 
is a map of $\delta$-functors.
Moreover it is compatible with cup products: 
$\phi_{O_{n,n},F\otimes F'}(x\cup y)= \phi_{O_{n,n},F}(x)\cup
\phi_{O_{n,n},F'}(y)$.
\end{lemma}

{\bf Step 2: $F\mapsto \H^*(O_{n,n}, F_n)$ is a universal $\delta$-functor.} We want to prove that $\H^*(O_{n,n}, F_n)$ vanishes in positive cohomological degree when $F$ is an injective of $\P$.  But the case of the orthogonal group is slightly different from the general linear and symplectic cases. Define $SO_{n,n}$ as the kernel of the Dickson invariant, or equivalently as the kernel of the determinant if $2$ is invertible in $\k$ (see \cite[p. 348]{BKINV} or \cite{Chaput} for details). Then  we have an extension of group schemes:
$$SO_{n,n}\;\triangleleft\; O_{n,n} \twoheadrightarrow \mathbb{Z}/2\mathbb{Z}\;.$$
And $SO_{n,n}$ is a Chevalley group scheme. Now \cite[section 4.9 p.509]{AJ} gives vanishing results for $SO_{n,n}$:
\begin{lemma}\label{lm-vanish-SOnn}
Let $\k$ be a commutative ring and let $J$ be an injective in the category $\P$. Then $\H^i(SO_{n,n}, J_n)=0$ for $i>0$. 
\end{lemma}
\begin{proof} By lemma \ref{lm-inj-formecombinatoire}, it suffices to prove the statement for the injectives $I^d_{k}:V\mapsto S^d(V^{\oplus k})$, 
for $k,d\ge 0$.
By lemma \ref{lm-vanish-chevalley}, it suffices to prove the vanishing over a field $\k$. In that case, \cite[section 4.9 p.509]{AJ} yields a good filtration on $S^*((\k^{2n\,\vee})^{\oplus k})$, whence the result.
\end{proof}

But we want a vanishing result for the cohomology of $O_{n,n}$, not for $SO_{n,n}$.
The Lyndon-Hochschild-Serre spectral sequence \cite[Part I, Prop 6.6(3)]{Jantzen} yields a graded isomorphism
$$\H^*(\mathbb{Z}/2\mathbb{Z}, \H^0(SO_{n,n},J_n))\simeq \H^*(O_{n,n}, J_n)\;.$$ Here comes our restriction on the characteristic. If $2$ is invertible in $\k$, then $\mathbb{Z}/2\mathbb{Z}$ is linearly reductive (Maschke's theorem) hence has no cohomology, so we get:
\begin{lemma}\label{lm-van-Onn}
Assume $2$ is invertible in $\k$. Then for all $J$ injective in $\P$, and for all positive $i$, $\H^i(O_{n,n}, J_n)$ equals zero.
So $F\mapsto \H^*(O_{n,n}, F_n)$ is a universal $\delta$-functor.
\end{lemma}

\begin{remark}
If $2$ is not invertible in $\k$, then the finite group $\mathbb{Z}/2\mathbb{Z}$ may have non trivial cohomology, so the above argument does not work. In fact, not only the proof but also the statement of lemma \ref{lm-van-Onn} is false when $2$ is not invertible in $\k$. So our restriction on the characteristic is necessary. Indeed, consider the constant functor $\k\in \P$. Then $\k$ is injective in $\P$, and $\H^0(SO_{n,n},\k)=\k$, so $\H^*(O_{n,n}, \k)\simeq\H^*(\mathbb{Z}/2\mathbb{Z},\k)$. Take  $\k$ a field of characteristic $2$. Then $\H^i(\mathbb{Z}/2\mathbb{Z},\k)\simeq \k$ for all $i$, so $F\mapsto \H^*(O_{n,n}, F_n)$ is not universal.
\end{remark}

\begin{lemma}[Step 3]\label{lm-step3-Onn}
Let $F\in\P$ be a functor defined over a commutative ring $\k$. If $2n$
is greater than the total degree of $F$, then 
$\phi^0_{O_{n,n},F}$ is injective.
\end{lemma}
\begin{proof}
We use lemma \ref{lm-key-Spn} to produce a suitable epimorphism
$\widetilde{\theta}:P^{2d}_{(\k^{2n})^\vee}\twoheadrightarrow \Gamma^d(S^2)$,
so that $\phi^0_{O_{n,n},F}$ is the composite of a Yoneda isomorphism and the
injective map $\hom_\P(\widetilde{\theta},F)$.
\end{proof}

\begin{lemma}[Step 4] Let $F\in\P$ be a functor defined over a commutative ring $\k$. If $2n\ge \deg
F$, then $\phi^0_{O_{n,n},F}$
is an isomorphism.
\end{lemma}
\begin{proof}
As in the symplectic case, it suffices to prove surjectivity for
the functors $I^d_k(V)=S^d(V^{\oplus k})$, $d\ge 0,k\ge 1$. Using compatibility with cup
products, the proof reduces furthermore to proving that
$\phi^0_{O_{n,n},I^*_k}$ hits the generators of the invariant ring
$\H^0(O_{n,n}, (I^*_k)_n)=\H^0(O_{n,n}, S^*(({\k^{2n}}^\vee)^{\oplus k}))$,
for all $k\ge 1$.

Let $b_n$ be the bilinear form associated to $q_n$. By \cite[Thm 5.6]{DCP}, a set of generators is given by
 the homogeneous polynomials $(i|j)_{1\le
i< j\le k}$ of degree
$2$, which send $(v_1,\dots,v_k)$ to $b_n(v_i,v_j)$, and by the $(i|i)_{1\le
i\le n}$ of degree $2$, which send $(v_1,\dots,v_k)$ to $q(v_i)$.
 For $1\le i\le k$, 
let $\iota_{i,i}$ be the inclusion of $V$ into the $i$-th term of $V^{\oplus
k}$. Then $\phi_{O_{n,n},I^2_k}^0$ sends $S^2(\iota_{i,i})$ to $(i|i)$. Assume
now that $1\le i<j\le k$. Denote by $\iota_{i,j}$ the inclusion of $V\oplus V$
in the $i$-th and the $j$-th terms of $V^{\oplus k}$. Let also $\rho$ be the 
composite $S^2(V)\to V\otimes V \to S^2(V\oplus V)$, where the second map is
induced by the exponential isomorphism for $S^2$. If we take 
$V={\k^{2n}}^\vee$, then $\rho$ sends $q_n$ to the sum $\sum_{i=1}^n 
(e_i^\vee,0)(0,e_{i+n}^\vee)+ \sum_{i=1}^n(e_{i+n}^\vee,0)(0,e_i^\vee)$, which
is nothing but the polynomial $\k^{2n}\oplus \k^{2n}\to \k$, $(v,w)\mapsto
b_n(v,w)$. Thus, $\phi_{O_{n,n},I^2_k}^0$ sends the natural transformation
$S^2(\iota_{i,j})\circ \rho$ to $(i|j)$. This concludes the proof.
\end{proof}

\begin{theorem}[The $O_{n,n}$ case]\label{thm-Onncase}
Let $\k$ be a commutative ring, and let $n$ be a positive integer.
For all $F\in\P$
we have a $*$-graded map,
natural in $F$:
$$\phi_{O_{n,n},F}:\Ext^*_{\P}(\Gamma^\star(S^2),F)\to \H^*(O_{n,n},F_n)\;. $$
The map $\phi_{O_{n,n},F}$ is compatible with cup products:
$$\phi_{O_{n,n},F\otimes F'}(x\cup y)= \phi_{O_{n,n},F}(x)\cup
\phi_{O_{n,n},F'}(y)\;.$$
Moreover, if $2n$ is greater or equal to the degree of $F$ and if $2$ is
invertible in $\k$,
then $\phi_{O_{n,n},F}$ is an isomorphism. 
\end{theorem}

%+++++++++++++++++++++++++++++++++++++++++++++++++++++

\section{Products of classical groups and cohomological
stabilization}\label{sec-4}

In this section, we use K\"unneth formulas to extend the link
between functor cohomology and rational cohomology to products of classical
groups. We also prove the cohomological
stabilization property for classical groups and their products.

\subsection{External tensor products and K\"unneth isomorphisms}
Let $\k$ be a commutative ring and for $i=1,2$, let $\A_i$ be a finite product
of copies of $\V_\k$ or its opposite category. If $F_i\in\P_{\A_i}$, $i=1,2$,
their external tensor product $F_1\boxtimes F_2$ is the functor sending
$(X,Y)$ to $F(X)\otimes F(Y)$. This yields a biexact bifunctor
$$-\boxtimes -:\P_{\A_1}\times\P_{\A_2}\to \P_{\A_1\times\A_2}\;.$$
Let us give some well-known \cite{SFB,FFSS,FF}
properties of external tensor products:
\begin{lemma}\label{lm-kunneth} 
For all $X_i\in\A_i$, external tensor product of standard injectives satisfy the
formula $I^*_{X_1}\boxtimes I^*_{X_2}\simeq I^*_{(X_1,X_2)}$, and we
have a commutative diagram:
$$
\xymatrix{
\hom_{\P_{\A_1}}(F_1,I^*_{X_1})\otimes
\hom_{\P_{\A_2}}(F_2,I^*_{X_2})\ar[r]^-{-\boxtimes-}\ar[d]^{\simeq}&
\hom_{\P_{\A_1\times\A_2}}(F_1\boxtimes F_2,I^*_{(X_1,X_2)})\ar[d]^{\simeq}\\
F_1(X_1)^\vee\otimes F_2(X_2)^\vee\ar@{=}[r]& F_1(X_1)^\vee\otimes
F_2(X_2)^\vee\;,
}
$$
where the vertical arrows are Yoneda isomorphisms.
Moreover, if $\k$ is a field, then for all $F_1,F_2,G_1,G_2$, $-\boxtimes -$ induces an
isomorphism:
$$\Ext^*_{\P_{\A_1}}(F_1,G_1)\otimes \Ext^*_{\P_{\A_2}}(F_2,G_2)\simeq
\Ext^*_{\P_{\A_1\times\A_2}}(F_1\boxtimes F_2,G_1\boxtimes G_2)\;. $$ 
\end{lemma}

Representations of algebraic groups have a similar external product. For
$i=1,2$, let $G_i$ be an algebraic group over $\k$ and let $M_i$ be a
$G_i$-module. The $\k$-module $M_1\otimes M_2$ is naturally endowed with the structure of a
$G_1\times G_2$-module, which we denote by $M_1\boxtimes M_2$. A computation
on the Hochschild complex gives:
\begin{lemma}\label{lm-kunnethgroupes}For
$i=1,2$, let $G_i$ be a flat algebraic group over $\k$ and let $M_i$ be a
$\k$-flat acyclic $G_i$-module.
Assume furthermore that $\H^0(G_1,M_1)$ is $\k$-flat. Then $M_1\boxtimes M_2$
is an acyclic $G_1\times G_2$-module and we have an isomorphism:
$$\H^0(G_1,M_1)\otimes \H^0(G_2,M_2)\simeq \H^0(G_1\times G_2,M_1\boxtimes
M_2)\;.$$
Moreover, if $\k$ is a field,  then for all $M_1,M_2$, $-\boxtimes -$ induces
an isomorphism:
$$\H^*(G_1,M_1)\otimes \H^*(G_2,M_2)\simeq \H^*(G_1\times G_2,M_1\boxtimes
M_2)\;.$$ 
\end{lemma}

\subsection{Application to products of classical groups}
\label{subsec-productsclassicalgps}

Let $\k$ be a commutative ring. We want to extend the results of section
\ref{sec-3} to algebraic groups $G_n$ over $\k$ which are finite products of
classical groups. 

To deal with products, we need some notations. Assume that $G_n=\prod_{i=1}^N
G^i_n$, where $G^i_n=GL_n, Sp_n$ or $O_{n,n}$. To each factor $G_n^i$ we
associate a category $\A_i$, a `characteristic functor'
$F_{G^i}\in\P_{\A_i}$ of degree two, a representation $V_i^n\in\A_i$ and an invariant
 $e^n_i\in F_{G^i}(V^n_i)$ like in
section \ref{sec-3}:
\begin{center}
\begin{tabular}{|c|ccc|}
\hline
$G^i_n$ & $GL_n$ & $Sp_n$ & $O_{n,n}$ \\
\hline
$\A_i$ & $ \V_\k^\op\times\V_\k$ & $\V_\k$ & $\V_\k$\\
$F_{G^i}$& $gl$ & $\Lambda^2$ & $S^2$\\
$V^n_i$  & $(\k^n,\k^n)$ & ${\k^{2n}}^\vee$ & ${\k^{2n}}^\vee$\\
$e_i^n$  & $\Id_{\k^n}$  & $\omega_n$   & $q_n$ \\
\hline
\end{tabular}
\end{center}
For all $d\ge 0$, let $\boxplus: \prod_{i=1}^N\P_{\A_i}\to \P_{\prod\A_i}$
be the functor induced by the direct sum. We define:
$$\A:=\textstyle\prod_i{\A_i}\,,\quad
V^n:=(V^n_i)\,,\quad F_G:=\boxplus_i F_{G^i} \,,\quad 
e^n:= (e^n_i)\,. $$

\begin{terminology}\label{terminology}
Let $G_n$ be a finite product of the $GL_n, Sp_n$ or $O_{n,n}$. We shall often
denote by $\P_G$ the category of strict polynomial functors with source $\A$
as above.
We refer to these functors as the functors `adapted to
$G_n$'. Indeed for all $n\ge 1$, 
since the $V_i^n$ have a structure of $G_i$-module, 
evaluation on $V^n\in\A$ yields a functor 
$$\P_G\to\text{$G_n$-mod}\,,\quad F\mapsto F_n:= F(V^n)\;.$$
\end{terminology}

\begin{example}
If $G_n=GL_n\times Sp_n$, then $\P_G$ is the category of strict polynomial
functors with source $\V_\k^\op\times\V_\k\times\V_\k$. For all $n\ge 1$ and
any functor $F$ adapted to $G_n$, the rational $G_n$-module $F_n$ equals
$F(\k^n,\k^n,(\k^{2n})^\vee)$ as a $\k$-module, and an element $(g,s)\in G_n$
acts by the formula $v\mapsto F(g^{-1},g,s)(v)$.  
\end{example}

\begin{theorem}\label{thm-generique}Let $\k$ be a commutative ring, let $n$ be a positive integer
and let $G_n$ be a finite product of the algebraic groups (over $\k$)
$GL_n$, $Sp_n$ and $O_{n,n}$. For all $F\in\P_G$
 we have a $*$-graded map,
natural in $F$:
$$\phi_{G_n,F}:\Ext^*_{\P_G}(\Gamma^\star(F_G),F)\to \H^*(G_n,F_n)\;. $$
The map $\phi_{G_n,-}$
is compatible with cup products:
$$\phi_{G_n,F\otimes F'}(x\cup y)= \phi_{G_n,F}(x)\cup
\phi_{G_n,F'}(y)\;.$$

Assume that $2n$ is greater or equal to the degree of $F$. If one of the factors of
$G_n$ equals $O_{n,n}$, 
assume furthermore that $2$ is invertible in $\k$. 
Then $\phi_{G_n,F}$ is an isomorphism. 
\end{theorem}

\begin{proof}
Once again we use a $\delta$-functor argument.
{\bf Step 1.} We build $\phi_{G_n,F}$. First for all $d$, $\Gamma^d(F_G)$ is
homogeneous of degree $2d$, so by homogeneity it suffices to do the
construction for a degree $2d$ homogeneous functor $F$. The element
$e^n\in(F_G)_n= F_G(V_n)$ is $G_n$-invariant, so we have a $G_n$-equivariant
map $\iota^d:\k\to \Gamma^d((F_G)_n)$, $\lambda\mapsto \lambda (e^n)^{\otimes d}$. 

Now a class in $x\in\Ext^i_{\P_\A}(\Gamma^\star(F_G),F)$ is represented by an
extension $F\hookrightarrow \dots\twoheadrightarrow \Gamma^{2d}(F_G)$. We
define $\phi_{G_n,F}(x)$ as the class of the extension obtained by evaluation
on $V^n$ and pullback by $\iota^d$.
Following the proof of lemma \ref{lm-step1-GLn}, 
we check that $\phi_{G_n,F}(x)$ is a map of $\delta$-functors, compatible
with cup products.

{\bf Step 2.} Using the exponential isomorphism for $S^*$ and lemma
\ref{lm-inj-formecombinatoire}, we see that the injectives of $\P_\A$ are
(direct summands in) finite direct sums of injectives of the form
$\boxtimes_{i=1}^N I^{d_i}$, where $I^{d_i}$ is either an injective of the
form $I^{d_i}_{k,\ell}$ or $I^{d_i}_k$, according to the fact that
$\A_i=\V^\op_\k\times \V_\k$ or $\V_\k$.

Using this and lemma \ref{lm-kunnethgroupes}, we obtain that $F\mapsto
\H^*(G_n,F_n)$ is a universal $\delta$-functor. By definition, $F\mapsto
\Ext^*_{\P_\A}(\Gamma^\star(F_G),F)$ is also a universal $\delta$-functor.

{\bf Step 3.} So to finish the proof, it suffices to prove that $\phi_{G_n,F}^0$ is an
isomorphism if $2n\ge d$, where $d$ is the degree of $F$. 
By left exactness of $F\mapsto
\H^0(G_n,F_n)$ and $F\mapsto
\hom_{\P_\A}(\Gamma^\star(F_G),F)$, it suffices to prove the isomorphism 
for $F=\boxtimes_{i=1}^N I^{d_i}$, with $\sum d_i\le d$. But in that case we
have a commutative diagram:
$$\xymatrix{
\bigotimes_i
\hom_{\P_{\A_i}}(\Gamma^\star(F_{G^i}),I^{d_i})\ar[rr]^-{\otimes\phi_{G^i_n,I^{d_i}}^0}
\ar[d]^-{\simeq}
&&
\bigotimes_i \H^0(G^i_n,(I^{d_i})_n)\ar[d]^-{\simeq}\;.\\
\hom_{\P_{\A}}(\boxtimes_{i=1}^N\Gamma^\star(F_{G^i}),\boxtimes_{i=1}^N I^{d_i})
\ar[d]^-{\simeq}&&
\H^0(G_n,(\boxtimes_{i=1}^N I^{d_i})_n)\\
\hom_{\P_{\A}}(\Gamma^\star(F_{G}),\boxtimes_{i=1}^N I^{d_i})
\ar[rru]_-{\phi_{G_n,\boxtimes_{i=1}^N I^{d_i}}^0}
&&
}$$
Since for all $i$, $I^{d_i}$ is a functor of degree $d_i\le d\le 2n$, we
deduce that the horizontal map of the diagram, hence 
$\phi^0_{G_n,\boxtimes_{i=1}^N I^{d_i} }$, is an isomorphism. This
concludes the proof.
\end{proof}

\subsection{Cohomological stabilization}\label{subsec-cohom-stab}

We keep the notations of paragraph \ref{subsec-productsclassicalgps}. In
particular, $\k$ is a commutative ring and $G_n=\prod_i G^i_n$,
where the $G^i_n$ are copies of the algebraic groups
$GL_n,Sp_n$ or $O_{n,n}$ over $\k$, and $V^n$ denotes the tuple $(V_i^n)$
where the $V^n_i$ are $G^i_n$-modules (or pairs of $G_n^i$-modules in the
general linear case). 
 
Let $n\le m$ be two positive integers. For all $i$ we have a standard
embedding $\iota_i :G^i_n\hookrightarrow G^i_m$ and a standard 
$G^i_n$-equivariant map $\pi_i: V^m_i\twoheadrightarrow V^n_i$. 
Let $\iota=\prod\iota_i$ and $\pi=\prod\pi_i$. The pair $(\iota,\pi)$ induces
a morphism in rational cohomology:
$$\phi_{n,m}:=\H^*(G_m,F(V^m))\xrightarrow[]{\iota^*}
\H^*(G_n,F(V^m))\xrightarrow[]{F(\pi)_*} \H^*(G_n,F(V^n))\;.$$
Now theorem
\ref{thm-generique} implies:
\begin{corollary}\label{cor-stab}
Let $\k$ be a commutative ring, let $n$ be a positive integer
and let $G_n$ be a finite product of copies of
$GL_n$, $Sp_n$ or $O_{n,n}$. Let $F\in\P_G$ be a degree $d$ functor adapted to
$G_n$. Let $n,m$ be two positive integers such that $2m\ge 2n\ge d$. If 
the orthogonal group appears as one of the factors of $G_n$, assume
furthermore that $2$ is invertible in $\k$. Then the morphism
$$\phi_{n,m}:\H^*(G_m,F_m)\xrightarrow[]{\simeq} \H^*(G_n,F_n)$$
is an isomorphism.
\end{corollary}
\begin{proof}
We check that $(F_G)_m\xrightarrow[]{F_G(\pi)} (F_G)_n$ sends $e^m$ to $e^n$.
Thus $ \phi_{G_n,F}=\phi_{n,m}\circ\phi_{G_m,F}$, and we apply theorem
\ref{thm-generique}.
\end{proof}

\begin{remark}
Corollary \ref{cor-stab} is a good illustration of the differences between our
methods for classical algebraic groups and
the methods of \cite{FFSS,DjamentVespa} where classical groups over
finite fields are considered as finite groups.  
Indeed, in our case the cohomological stabilization is a byproduct of the
proof, whereas in the finite group case it is needed as 
an input for the proof.   
\end{remark}

%\begin{remark}
%By a result of Cline, Parshall, Scott and Van der Kallen \cite{CPSVdK}, if the
%rational cohomology $\H^*_\rat(G_n,M^{(r)})$ computes. 
%\end{remark}

\begin{notation}
If $G_n$ is a product of copies of $GL_n$, $Sp_n$ or $O_{n,n}$, and if $F$ is
a strict polynomial functor adapted to $G_n$, we denote by 
$\H^*(G_\infty,F_\infty)$ the stable value of the $\H^*(G_n,F_n)$.
\end{notation}

\section{Products and coproducts on functor cohomology}
\label{sec-prodetcoprod}

In this section, $\k$ is a field (we need this condition because we use 
in many places 
the K\"unneth isomorphism of lemma \ref{lm-kunneth}). We study product and
coproduct structures which arise on functor cohomology
$\Ext^*_{\P_\A}(E^*,-)$. Our purpose is to generalize and
clarify the tools of \cite[Lemma 1.10 and 1.11]{FFSS}.

Sections \ref{subsec-hopfalg}, \ref{subsec-hopfmonstr} and \ref{subsec-sum-diag-adj} are introductory. We recall the definition of `Hopf algebra functor', we introduce the notion of `Hopf monoidal functor' (which is useful to describe structures on strict polynomial functors $E^*$, as well as the
structures on functor cohomology $\Ext^*(E^*,-)$).
Then we recall a classical tool of functor categories \cite{FFSS,FF},
namely, the sum-diagonal adjunction. This tool is the key for the
existence of coproducts and more generally of Hopf monoidal structures on
functor cohomology.

With these tools at our disposal, we make an attempt to classify the Hopf monoidal structure which may arise on extension groups of the form $\Ext^*(E^*,-)$. To be more specific, we give in 
section \ref{subsec-class1} bijections between:
\begin{itemize}
\item[(1)] Hopf algebra structures on $E^*$ (denoted by  $(m_E,1_E,\Delta_E,\epsilon_E)$),
\item[(2)] Hopf monoidal structures on $E^*$ (denoted by  $(\mu,\eta,\lambda,\epsilon)$),
\item[(3)] Hopf monoidal structures on $\hom(E^*,-)$.
\end{itemize}
Taking injective resolutions, these structures yield Hopf
monoidal structures on $\Ext^*(E^*,-)$.

In fact, we don't need the classification of Hopf monoidal structures on $\Ext^*(E^*,-)$ for our applications. We only need theorem \ref{thm-Hopfmon} which states that a Hopf algebra structure on $E^*$ induces a Hopf monoidal structure on $\Ext^*(E^*,-)$, and gives two equivalent descriptions of the external cup product. But \cite[Lemma 1.10 and 1.11]{FFSS} use a superflous hypothesis (the functors need not be exponential), and also has a sign problem, so we thought it was worth clarifying the situation.

\subsubsection*{Convention on gradings} If $n\ge 0$ is an integer an $n$-graded
object is a family of objects indexed by $n$-tuples of nonnegative integers
(Thus, a $0$-graded object is a family indexed by the empty tuple `$(\,)$', in
other words a $0$-graded object is
just a non-graded object). 
We denote $n$-gradations by a single `$*$' sign. 
If $*=(i_1,\dots,i_n)$ and $\star=(j_1,\dots,j_n)$, then $*+\star$ is the
tuple $(i_1+j_1,\dots,i_n+j_n)$, $*\star=(i_1j_1,\dots,i_nj_n)$ and $|*|$ is
the integer $\sum i_k$ (in particular $|(\,)|=0$). 

We often drop the gradings and write $X$ for a multigraded object instead of
$X^*$ when no confusion is possible.

\subsection{Hopf algebra functors}\label{subsec-hopfalg}
In this section and in the
remainder of the paper, we define Hopf algebras as in \cite{ML}, that is
without requiring an antipode.

Thus if $F^*$ is a $n$-graded functor from a category $\C$ to the category of $\k$-vector spaces,  a 
`$n$-graded Hopf algebra structure on
$F^*$' is a tuple $(m_F,1_F,\Delta_F,\epsilon_F)$ of $n$-graded
natural maps
$$F^*(X)^{\otimes 2}\xrightarrow[]{m_F} F^*(X)\,,\;
\k\xrightarrow[]{1_F} F^*(X)\,,\;
F^*(X)\xrightarrow[]{\Delta_F} F^*(X)^{\otimes 2}\,,\;
F^*(X)\xrightarrow[]{\epsilon_F}\k\,,$$
such that for all $X\in\C$, $F^*(X)$ is an $n$-graded Hopf algebra.

\subsection{Hopf monoidal functors}\label{subsec-hopfmonstr}
Let $\k$ be a field, and let $(\C,\square,e)$ be a symmetric monoidal category
\cite[VII.7]{MLCat}. We consider the category $\kvect$ of $\k$-vector spaces
as a symmetric monoidal category, with monoidal product the usual tensor
product over $\k$. We fix an $n$-graded functor
$F^*:\C\to\kvect$. 
We regard $\k$  as an $n$-graded constant functor concentrated in
degree $(0,\dots,0)$.
A \emph{$n$-graded monoidal structure} on
$F^*$ is a pair $(\mu,\eta)$ of $n$-graded maps:
\begin{align*}
&\mu: \textstyle 
F^*(X)\otimes F^\star(Y)\to F^{*+\star}(X\square Y)\;,
&\eta: \k \to F^*(e) \;,
\end{align*} 
which satisfy an associativity and a unit condition
\cite[XI.2]{MLCat}. By reversing the arrows, one obtains the notion of an
$n$-graded comonoidal structure $(\lambda,\epsilon)$ on $F^*$. A monoid in
$\C$ is an object equipped with a multiplication $M\square M\to M$ and a unit
$e\to M$ satisfying an associativity and a unit condition \cite[VII.3]{MLCat}.
By reversing the arrows one gets the definition of a comonoid in $\C$. The
following lemma is straightforward from the axioms:
\begin{lemma}\label{lm-monoidal-funct-monoids}
Let $F^*:\C\to\kvect$ be an $n$-graded monoidal functor and let 
$M$ be a monoid in $\C$. The maps:
\begin{align*}
&
F^*(M)\otimes F^\star(M)\xrightarrow[]{\mu} 
F^{*+\star}(M\square M)
\xrightarrow[]{} F^{*+\star}(M)\;,
&\k\xrightarrow[]{\eta} F^*(e) \xrightarrow[]{} F^*(M)\;,
\end{align*}
make $F^*(M)$ into an $n$-graded algebra. In particular, $F^*(e)$ is an
$n$-graded algebra. Similarly, an $n$-graded
comonoidal functor sends a comonoid to an $n$-graded coalgebra, and $F^*(e)$ is
an $n$-graded coalgebra.
\end{lemma}

Let $\tau$ be the isomorphism $X\otimes Y\xrightarrow[]{\simeq}
 Y\otimes X$, and let
$\tau^*$ be its $n$-graded version, which sends the tensor product $x\otimes y$
of an element $x$ of $n$-degree $*$ and an element $y$ of $n$-degree $\star$
to $(-1)^{|*\star|}y\otimes x$.  

\begin{definition}\label{def-Hopfmonstr}
A \emph{$n$-graded Hopf
monoidal structure} on $F^*$ is a tuple $(\mu,\lambda,\eta,\epsilon)$ such
that:
\begin{enumerate}
\item[(0)] $(\mu,\eta)$ is an $n$-graded monoidal structure on $F^*$ and
$(\lambda,\epsilon)$ is an $n$-graded comonoidal structure on $F^*$.
\item[(1)] $\eta:\k\to F^*(e)$ is a morphism of $n$-graded coalgebras.
\item[(2)] $\epsilon:F^*(e)\to\k$ is a morphism of $n$-graded algebras.
\item[(3)] The following diagram commutes~:
$$\xymatrix{
F(X\square Y) \otimes F(Z\square T) 
\ar[r]_-{\lambda\otimes
\lambda}\ar[d]^{\mu}&
F(X) \otimes F(Y) \otimes F(Z) \otimes F(T)
 \ar[d]_{F(X)\square \tau^{*}\square F(T)}\\
F(X\square Y\square Z\square
T)\ar[d]^{F(X\square\tau\square T)}&
F(X) \otimes F(Z) \otimes F(Y)\otimes F(T) 
\ar[d]_{\mu\otimes \mu}
\\
F(X\square Z\square Y\square T)\ar[r]^-{\lambda}
& F(X\square Z)\otimes F(Y\square T)\;.
}$$
\end{enumerate}
\end{definition}
A Hopf monoid in $\C$ is an object $M$ which is both
a monoid and a comonoid, and such that (1) the unit $e\to M$ is a map of
comonoids, (2) the counit $M\to e$ is a map of monoids, and (3) the
comultiplication $M\to M\square M$ is a map of monoids ($M\square M$ can be made
into a monoid in the obvious way because $\C$ is \emph{symmetric} monoidal).
For example, a Hopf monoid in $\kvect$ is nothing but a Hopf algebra.
With this definition
we immediately obtain the Hopf analogue of lemma
\ref{lm-monoidal-funct-monoids}:
\begin{lemma}\label{lm-Hopf-monoidal-funct-monoids}
Let $F^*:\C\to\kvect$ be an $n$-graded Hopf monoidal functor and let 
$M$ be a Hopf monoid in $\C$. The monoid and the comonoid structures on
$F^*(M)$ given by lemma \ref{lm-monoidal-funct-monoids}
make $F^*(M)$ into an $n$-graded Hopf algebra. In particular, $F^*(e)$ is an
$n$-graded Hopf algebra.
\end{lemma}

We finish the presentation by giving examples. 

\begin{lemma}\label{lm-ExpHopf}
Let $(\C,\square,e)$ be a symmetric monoidal category and let
$(F^*,\mu,\eta)$ be an $n$-graded symmetric monoidal functor from $\C$ to $\kvect$, 
such that $F^*$ has finite dimensional values, and for all $X,Y$, $\mu_{X,Y}:F^*(X)\otimes F^*(Y)\to F^*(X\square Y)$
is an isomorphism. We have:
\begin{itemize}
\item[(a)] the unit $\eta$ induces an isomorphism
 $\k\xrightarrow[]{\simeq}F^{(0,\dots,0)}(e)$.
\item[(b)] Let $\epsilon$ denote the composite $F^*(e)\twoheadrightarrow
F^{(0,\dots,0)}(e)\simeq \k$. Then $(\mu,\eta,\mu^{-1},\epsilon)$ is an
$n$-graded Hopf monoidal structure on $F^*$ if and only if for all $Y,Z$, the
following diagram commutes:
$$\xymatrix{
F^*(Y)\otimes F^*(Z)\ar[d]^-{\tau^*}\ar[r]^-{\mu_{Y,Z}}
& F^*(Y\square Z)\ar[d]^-{F^*(\tau)}\\
F^*(Z)\otimes F^*(Y)\ar[r]^-{\mu_{Z,Y}}& F^*(Z\square Y)\;.
}$$
\end{itemize}
\end{lemma}
\begin{proof}
(a) Because of the unit axiom for $(F^*,\mu,\eta)$, we know that $\eta$ is
injective. Since the $\lambda_{X,Y}$ are isomorphisms, we have
$F^{(0,\dots,0)}(e)\simeq F^{(0,\dots,0)}(e\square e)\simeq 
F^{(0,\dots,0)}(e)^{\otimes 2}$. Using finite dimension of these 
vector spaces, we deduce
that  $F^{(0,\dots,0)}(e)$ is one dimensional, whence the result.

(b) A trivial verification shows that $(\mu,\eta,\mu^{-1},\epsilon)$ satisfies
axioms (0-2) of definition \ref{def-Hopfmonstr} (without assuming that the
diagram commutes). Now we check that axiom (3) is satisfied if and
only if the diagram commutes. To prove the `only if' part, evaluate axiom (3) 
on $X=T=e$. To prove the `if' part, tensor the commutative diagram on the left
by $F^*(X)$, on the right by $F^*(Y)$ and use the associativity of $\mu$. 
\end{proof}

\begin{example}
Let $(\C,\square,e)$ be the category
$(\V_\k,\oplus,0)$ of finite
dimensional vector spaces. For all $V\in\V_\k$ we consider
the divided powers $\Gamma^*(V)$ with $\Gamma^d(V)$ in degree $2d$.
Then the exponential isomorphism (cf \S\ref{subsec-nota}) $\Gamma^*(V)\otimes
\Gamma^*(W)\xrightarrow[]{\simeq} \Gamma^*(V\oplus W)$ and 
the unit $\k=\Gamma^0(0)=\Gamma^*(0)$ satisfy the hypothesis of lemma
\ref{lm-ExpHopf}. Thus, they induce a graded Hopf monoidal structure on
$\Gamma^*$. Similarly, $S^*(V)$
(with $S^d(V)$ placed in degree $2d$) and $\Lambda^*(V)$ 
(with $\Lambda^d(V)$ placed in degree $d$) have a Hopf monoidal structure
defined by the exponential isomorphism.
\end{example}

\begin{remark}\label{rk-degrees}
We warn the reader that for $\Gamma^*(V)$ with $\Gamma^d(V)$ in degree $d$, the
above structure is \emph{not} a Hopf monoidal structure (axiom (3) fails). For
an analogous reason, $\Gamma^*(V)$ with $\Gamma^d(V)$ in degree $d$ is not a
graded Hopf algebra. In particular, \cite[Lemma 1.10]{FFSS} is
false as stated, and our lemma \ref{lm-ExpHopf}(b) indicates the missing
hypothesis. To be more specific, 
the only `Hopf exponential functors' which satisfy the conclusion
of \cite[Lemma 1.10]{FFSS} are the `skew commutative' ones.

In general, axiom (3) of definition \ref{def-Hopfmonstr} 
is a constraint for the gradings. 
For example if $F^*$ is a graded Hopf monoidal functor, by `forgeting' 
the grading, one does not obtain a non graded Hopf monoidal functor (except if
$\k$ has characteristic two or if $F^*$ is concentrated in even degrees). The
same defect arises for multigraded Hopf algebras and lemma
\ref{lm-Hopf-monoidal-funct-monoids} explains the link.
\end{remark}

\subsection{The sum-diagonal adjunction}\label{subsec-sum-diag-adj}

General statements about adjunction isomorphisms in functor categories are
given for example in \cite{Pira}. We sketch here the arguments in our
specific case and give explicit formulas.

As usual, $\A$ is a finite product of copies of $\V_\k$ and 
$\V^{\op}_\k$. In particular, $\A$ is an additive category.
The diagonal
functor $D:\A\to \A\times\A$, $X\mapsto (X,X)$ is left adjoint to the sum functor
$\Oplus:\A\times\A\to \A$, $(X,Y)\mapsto X\oplus Y$. 
To be more specific, if $\delta_2$ is the diagonal $V\to
V\oplus V$, $v\mapsto (v,v)$ and $\pr_i:V_1\oplus V_2\to V_i$, $i=1,2$, 
is the projection onto the $i$-th factor, we easily check that the unit, 
resp. the counit, of this adjunction equals:
$$\delta_2:\Id_\A\to \Oplus\circ D  \;,\text{ resp.}\quad
(\pr_1,\pr_2):D\circ \Oplus\to \Id_{\A\times\A}\;.$$
Precomposition by $D$ and $\Oplus$ yields adjoint functors 
$-\circ \Oplus:
\P_\A\to \P_{\A\times\A}$ and
$-\circ D:\P_{\A\times\A}\to
\P_\A$. Let's be more explicit. We denote by
$F(\Oplus)$ and $G(D)$ 
the functors $F$ and $G$ precomposed by $\Oplus$ and $D$. Then the
adjunction isomorphism is given by:
$$\begin{array}[t]{ccc}
\hom_{\P_{\A\times\A}}(F(\Oplus),G)
&\xrightarrow[]{\simeq} 
&\hom_{\P_\A}(F,G(D))\\
f&\mapsto& f(D)\circ F(\delta_2)
\end{array},
$$ 
with inverse $g\mapsto G(\pr_1,\pr_2)\circ g(\Oplus)$. For all $X,Y\in\A$ we
have $I^*_{(X,Y)}(D)\simeq I^*_{X\oplus Y}$. Hence $-\circ D$ preserves the
injectives. 
One easily computes:
\begin{lemma}\label{lm-adj-inj}
Let $X,Y\in\A$. Denote by $\pr_X,\pr_Y$ the projections of $X\oplus Y$ onto
$X,Y$. Then we have a commutative diagram:
$$\xymatrix{
\hom_{\P_{\A\times\A}}(K,I^*_{(X,Y)})\ar[r]^-{-\circ D}\ar[d]^{\simeq}&
\hom_{\P_\A}(K(D),I^*_{X\oplus Y})\ar[d]^{\simeq}\\
K(X,Y)^\vee\ar[r]^-{K(\pr_X,\pr_Y)^\vee}& K(X\oplus Y,X\oplus Y)^\vee \;,
}$$
in which the vertical arrows are Yoneda isomorphisms. As a consequence,
the adjunction fits into a commutative diagram, in which
the vertical arrows are Yoneda isomorphisms:
$$\xymatrix{
\hom_{\P_{\A\times\A}}(F(\Oplus),I^*_{(X,Y)})\ar[r]^-{\alpha}\ar[d]^{\simeq}&
\hom_{\P_\A}(F,I^*_{X\oplus Y})\ar[d]^{\simeq}\\
F(X\oplus Y)^\vee\ar@{=}[r]& F(X\oplus Y)^\vee \;.
}$$
\end{lemma}

Since $-\circ D$ preserves the injectives, we may take injective resolutions
to obtain:
\begin{lemma} For all $F\in\P_\A$ and all $G\in\P_{\A\times\A}$, 
there is an isomorphism, natural in $F,G$:
$$\alpha: \Ext^*_{\P_{\A\times\A}}(F(\Oplus),G)\xrightarrow[]{\simeq}
\Ext^*_{\P_\A}(F,G(D))\;.
$$
\end{lemma}
\begin{remark}\label{rk-2adj}
The functors $D$ and $\Oplus$ are adjoint on both sides. Using that $D$ is
right adjoint of $\Oplus$ one can get another adjunction isomorphism:
$\beta:\Ext^*_{\P_\A}(G(D),F)\simeq 
\Ext^*_{\P_{\A\times\A}}(G, F(\Oplus))
$. We don't use this latter isomorphism in this section.
\end{remark}

\subsection{Hopf monoidal structures on functor cohomology}
\label{subsec-class1}
In this paragraph, $\k$ is a field and we fix an $n$-graded functor
$E^*\in\P_\A$.
To avoid cumbersome notations, we drop the `$\P_\A$' index on $\hom$ or
$\Ext$-groups,  
as well as the grading on $E$ when no confusion is possible.   

We first examine structures which may equip $E^*$. First, $E^*$ may be endowed with a 
$n$-graded Hopf algebra structure on
$E^*$' is a tuple $(m_E,1_E,\Delta_E,\epsilon_E)$ of $n$-graded
natural maps
$$E(V)^{\otimes 2}\xrightarrow[]{m_E} E(V)\,,\;
\k\xrightarrow[]{1_E} E(V)\,,\;
E(V)\xrightarrow[]{\Delta_E} E(V)^{\otimes 2}\,,\;
E(V)\xrightarrow[]{\epsilon_E}\k\,,$$
such that for all $V\in\A$, $E^*(V)$ is an $n$-graded Hopf algebra.

On the other hand, the direct sum endows $\A$ with the structure of a symmetric
monoidal category. So we may also consider $n$-graded Hopf monoidal structures
on $E^*$, that is tuples $(\mu,\eta,\lambda,\epsilon)$ with $\mu: E(V)\otimes
E(W)\to E(V\oplus W)$, etc.
These two kinds of structure are equivalent:
\begin{lemma}\label{lm-A} To any $n$-graded Hopf monoidal
structure $(\mu,\eta,\lambda,\epsilon)$  on $E^*$, we associate
an $n$-graded Hopf algebra structure on $E^*$ defined as
follows:
$$m_E:E(V)^{\otimes 2}\xrightarrow[]{\mu_{V,V}} 
E(V\oplus V)\xrightarrow[]{E(\Sigma_2)}E(V), \quad 1_E: 
\k \xrightarrow[]{\eta}E(0)\xrightarrow[]{E(0)} E(V), $$
$$\Delta_E:E(V)\xrightarrow[]{E(\delta_2)}E(V\oplus V)
\xrightarrow[]{\lambda_{V,V}} 
E(V)^{\otimes 2}, \quad \epsilon_E: 
E(V)\xrightarrow[]{E(0)} E(0)\xrightarrow[]{\epsilon} \k. $$
This yields a bijection between the set of $n$-graded Hopf monoidal structures
$(\mu,\eta,\lambda,\epsilon)$  on $E^*$ and the set of
$n$-graded Hopf algebra structures
$(m_E,1_E,\Delta_E,\epsilon_E)$ on $E^*$.
%$$\left\{ 
%{
%\text{
%\begin{tabular}{c}
%$n$-graded
%Hopf monoidal\\ 
%structures $(\mu,\lambda,\eta,\epsilon)$
%on $E$
%\end{tabular}
%}
%}
%\right\}\xrightarrow[]{\simeq}
%\left\{ 
%{
%\text{
%\begin{tabular}{c}
%$n$-graded
%Hopf algebra\\
%structures $(m_E,\Delta_E,1_E,\epsilon_E)$
%on $E$
%\end{tabular}
%}
%}
%\right\}\;.$$

\end{lemma}
\begin{proof}
For all $V\in\A$, the sum $\Sigma_2:V\oplus V\to V$ and the diagonal
$\delta_2:V\to V\oplus V$ turn $V$ into a Hopf monoidal
object in $(\A,\oplus,0)$. Hence, by
lemma \ref{lm-Hopf-monoidal-funct-monoids}, $(m_E,1_E,\Delta_E,\epsilon_E)$
is actually a Hopf algebra
structure. To prove the bijection, 
we give its
inverse. If $V_i\in\A$, $i=1,2$, we denote by $\i_i$ the
inclusion of $V_i$ into $V_1\oplus V_2$ and by $\pr_i$ the projection of
$V_1\oplus V_2$ 
onto its $i$-th
factor. Now from a Hopf algebra structure $(m_E,1_E,\Delta_E,\epsilon_E)$
we define:
$$\textstyle\bigotimes E(V_i)\xrightarrow[]{\otimes E(\i_i)}
E(\Oplus V_i)^{\otimes 2} \xrightarrow[]{m_E}E(\Oplus V_i)\,,
\quad \k\xrightarrow[]{1_E}E(V)\xrightarrow[]{E(0)} E(0)\,,
$$
$$E(\Oplus V_i)\xrightarrow[]{\Delta_E}E(\Oplus V_i)^{\otimes 2}
\xrightarrow[]{\otimes E(\pr_i)}\textstyle\bigotimes E(V_i)\,,
\quad E(0)\xrightarrow[]{E(0)} E(V)\xrightarrow[]{\epsilon_E}\k\;.$$
A straightforward verification shows that this actually gives an $n$-graded
Hopf
monoidal structure on $E^*$, and that this yields the inverse of the map of
the lemma.
\end{proof}

\begin{lemma}\label{lm-B}
To any $n$-graded Hopf monoidal
structure $(\mu,\eta,\lambda,\epsilon)$  on $E^*$, we associate
an $n$-graded Hopf monoidal structure on $\hom(E,-)$ defined as
follows:
\begin{align*}
\hom_{\P_{\A}}(E,F)\otimes &\hom_{\P_{\A}}(E,G) \xrightarrow[\simeq]{\kappa}
\hom_{\P_{\A\times\A}}(E^{\boxtimes 2}, F\boxtimes G)
\\&\xrightarrow[]{\lambda^*}
\hom_{\P_{\A\times\A}}(E(\Oplus), F\boxtimes G)\xrightarrow[]{\alpha}
\hom_{\P_{\A}}(E, F\otimes G)\;,\\
\k= \hom(\k,\k)& \xrightarrow[]{\epsilon^*}\hom(E(0),\k)\to \hom(E ,\k)\,,
\end{align*}
\begin{align*}
\hom_{\P_{\A}}&( E  ,F\otimes G)\xrightarrow[\simeq]{\alpha^{-1}}
\hom_{\P_{\A\times\A}}(E(\Oplus),  F\boxtimes G)\\&
\xrightarrow[]{\mu^*} \hom_{\P_{\A\times\A}}(E^{\boxtimes 2}, F\boxtimes G)
\xrightarrow[\simeq]{\kappa^{-1}}
\hom_{\P_{\A}}(E,F)\otimes \hom_{\P_{\A}}(E,G)\,,
\\
\hom(E &,\k )\to \hom(E(0),\k )\xrightarrow[]{\eta^*}\hom(\k,\k)=\k
\,\;.
\end{align*}
This yields a bijection between the $n$-graded Hopf monoidal structures on
$E^*$ and the $n$-graded Hopf monoidal structures on $\hom(E^*,-)$.
\end{lemma}
\begin{proof}
By left exactness of the functor $\hom(E,-)$ and its tensor products, it
suffices to check the axioms when $F$ and $G$ are injective. Since the
injectives of $\P_\A$ are direct summands of (sums of) standard injectives, 
we can
furthermore assume that $F=I^*_X$ and $G=I^*_Y$, for $X,Y\in\A$. But in that
case, by lemmas \ref{lm-kunneth} and \ref{lm-adj-inj}, we have a commutative
diagram (in which the vertical arrows are Yoneda isomorphisms):
$$\xymatrix{
\hom(E,I^*_X)\otimes \hom(E,I^*_Y)\ar[r]\ar[d]^{\simeq}&
\hom(E,I^*_{X\oplus Y})\ar[d]^{\simeq}\\
E(X)^\vee \otimes E(Y)^\vee \ar[r]^{\lambda^\vee}& E(X\oplus Y)^\vee\;,
}$$
and also a similar diagram involving $\mu^\vee$. Using these two diagrams, we
easily check the
Hopf monoidal axioms for
$\hom(E,-)$ from the axioms satisfied by
$(\mu,\eta,\lambda,\epsilon)$.

Now it remains to show the bijection. Let us give the inverse.
If we have an $n$-graded 
Hopf monoidal structure on
$\hom(E^*,-)$, we may restrict to the standard injectives $I^*_X$, $X\in\A$.
By the Yoneda isomorphisms, we obtain a Hopf monoidal structure on $E^*$. The
diagrams mentioned above make it clear that this yields the inverse.
\end{proof}

We have proved:
\begin{theorem}\label{thm-mon-str-Hopf} Let $\k$ be a field,
 and let $E^*\in\P_\A$ be an $n$-graded functor. 
There are bijections between:
\begin{enumerate}
\item[(1)] The set of $n$-graded Hopf algebra structures on $E^*$.
\item[(2)] The set of $n$-graded Hopf monoidal structures on $E^*$.
\item[(3)] The set of $n$-graded Hopf monoidal structures on $\hom_{\P_\A}(E^*,-)$.
\end{enumerate}
\end{theorem}
Explicit formulas for the bijection between (2) and (1), and 
between (2) and (3) are
given in lemmas \ref{lm-A} and \ref{lm-B}. For further use, we also
need an explicit link
between the product 
$\hom(E,F)\otimes \hom(E,G)\to \hom(E,F\otimes G)$
and the $n$-graded Hopf algebra
structure of $E^*$.

\begin{lemma}[Key formula]\label{lm-2-description}
Let $(\mu,\eta,\lambda,\epsilon)$ be an $n$-graded 
Hopf monoidal structure on $E$, 
and let $(m_E,1_E,\Delta_E,\epsilon_E)$ be the associated
Hopf algebra structure (cf. lemma \ref{lm-A}). For any functors
$F_i$, $i=1,2$, the
following two composites are equal:
\begin{align}\nonumber
\hom(E,F_1)&\otimes \hom(E,F_2) \xrightarrow[\simeq]{\kappa}
\hom(E^{\boxtimes 2}, F_1\boxtimes F_2)
\\&\xrightarrow[]{\lambda^*}
\hom(E(\Oplus), F_1\boxtimes F_2)\xrightarrow[]{\alpha}
\hom(E, F_1\otimes F_2)\\
\hom(E,F_1)&\otimes \hom(E,F_2)\xrightarrow[]{\otimes }
\hom(E^{\otimes 2},\textstyle\bigotimes_i F_i)\xrightarrow[]{\Delta_E^*} 
\hom(E,\textstyle\bigotimes_i F_i)
\end{align}
\end{lemma}
\begin{proof}
We proceed in the same way as in the proof of lemma \ref{lm-B}.
By left exactness of $\hom(E,-)$ it suffices to prove the formula for standard
injectives $F_1=I^*_X$ and
$F_2=I^*_Y$. In that case, by lemmas \ref{lm-kunneth} and \ref{lm-adj-inj}, the first map identifies, through
Yoneda isomorphisms, with the map $\lambda^\vee: E(X)^\vee\otimes E(Y)^\vee\to 
E(X\oplus Y)^\vee$.
On the other hand, by lemmas \ref{lm-kunneth} and \ref{lm-adj-inj},
 the second map identifies with the composite:
$$E(X)^\vee\otimes E(Y)^\vee\xrightarrow[]{E(\pr_X)^\vee\otimes E(\pr_Y)^\vee}
E(X\oplus Y)^{\otimes 2}\xrightarrow[]{\Delta_E^\vee} E(X\oplus Y)^\vee\quad
(*)$$
Now by definition (lemma \ref{lm-A}) 
$\Delta_E=\lambda_{X\oplus Y,X\oplus Y}\circ E(\delta_2)$. By naturality of
$\lambda$, 
$(E(\pr_X)\otimes E(\pr_Y))\circ 
\lambda_{X\oplus Y,X\oplus Y}\circ E(\delta_2)$ equals $\lambda_{X,Y}\circ
E(\pr_X\oplus \pr_Y)\circ E(\delta_2)$ which in turn equals
$\lambda_{X,Y}$. Thus $(*)$ equals $\lambda^\vee$, and this concludes the
proof.
\end{proof}

Now we turn to Hopf monoidal structures on $\Ext$-groups. Let $E^*$ be an
$n$-graded functor in $\P_\A$ and suppose that $\hom(E,-)$ has an $n$-graded
monoidal structure $(\mu,\eta,\lambda,\epsilon)$. 
By taking injective resolutions, we obtain $(1+n)$-graded maps $\mu:
\bigotimes_i \Ext^*(E,F_i)\to \Ext^*(E,\bigotimes_i F_i)$, 
$\lambda:\Ext^*(E,\bigotimes_i F_i)\to \bigotimes_i \Ext^*(E,F_i)$, and
we also define $\eta:\k\to \hom(E,\k)=\Ext^*(E,\k)$ and 
$\epsilon:\Ext^*(E,\k)=\hom(E,\k)\to\k$. One easily sees that this
defines a $(1+n)$-graded Hopf monoidal structure on $\Ext^*(E^*,-)$ which
coincides with the Hopf monoidal structure of $\hom(E^*,-)$ in degree $(0,*)$.
Moreover, the resulting structure is a `$\delta$-Hopf monoidal
structure'
on $\Ext^*(E,-)$, that is, if we fix one of the two functors $F_i$, then $\mu$
and $\lambda$ become maps of $\delta$-functors. To sum up we have:

\begin{lemma}\label{lm-derive}
Let $\k$ be a field, and let $E^*\in\P_\A$ be an $n$-graded functor.
Derivation induces an injection:
$$ \left\{{
\text{
\begin{tabular}{c}
$n$-graded
Hopf monoidal\\
structures
on $\hom(E^*,-)$
\end{tabular}
}
}\right\}
\hookrightarrow
\left\{{
\text{
\begin{tabular}{c}
$(1+n)$-graded
$\delta$-Hopf monoidal\\
structures
on $\Ext^*(E^*,-)$
\end{tabular}
}
}\right\}\;.$$
\end{lemma}
\begin{remark}
The map of lemma \ref{lm-derive} is not a bijection in general. Indeed,
the condition of being a $\delta$-Hopf monoidal structure does not guaranty
that the structure is of derived type, i.e. obtained by applying a Hopf
monoidal
structure on $\hom(E,-)$ to injective resolutions. 
To be more specific, the $\delta$ condition guaranties that
the product $\mu$ is of derived type (cf. \cite[XII, Thm 10.4]{ML})
but in general it is not sufficient to guaranty that the coproduct $\lambda$ is of
derived type 
(see also \cite[Notes of XII.9]{ML}).
\end{remark}

Now lemmas \ref{lm-derive}, \ref{lm-2-description} and theorem
\ref{thm-mon-str-Hopf} yield:
\begin{theorem}\label{thm-Hopfmon}
Let $\k$ be a field, and let $E^*\in\P_\A$ be an $n$-graded functor,
endowed with a Hopf monoidal structure $(\mu,\eta,\lambda,\epsilon)$.
The functor cohomology cup product associated (cf. \S \ref{subsec-fctcohom}) 
to the comultiplication
$\Delta_E:E(V)\xrightarrow[]{E(\delta_2)} E(V\oplus V)\xrightarrow[]{\lambda}
E(V)^{\otimes 2}$  equals the composite:
\begin{align*}
\Ext^*(E,F)\otimes \Ext^*(E,G)\xrightarrow[\simeq]{\kappa}\Ext^*(E^{\boxtimes
2},F\boxtimes G)\xrightarrow[]{\lambda^*} &\Ext^*(E(\Oplus),F\boxtimes G)
\\&\xrightarrow[\simeq]{\alpha}\Ext^*(E,F\otimes G)\;.
\end{align*}
Together with the following unit, counit and coproduct, 
they make $\Ext^*(E^*,-)$ into a $(1+n)$-graded Hopf monoidal functor:
$$\k= \Ext^*(\k,\k)\xrightarrow[]{\epsilon^*} \Ext^*(E,\k)\,,\quad  
\Ext^*(E,\k)\xrightarrow[]{\eta^*} \Ext^*(\k,\k)= \k\,,$$
\begin{align*}
\Ext^*(E,F\otimes G)\xrightarrow[\simeq]{\alpha^{-1}}
\Ext^*(E(\Oplus),&F\boxtimes G)\xrightarrow[]{\mu^*} \Ext^*(E^{\boxtimes
2},F\boxtimes G)\\
&\xrightarrow[\simeq]{\kappa^{-1}}\Ext^*(E,F)\otimes \Ext^*(E,G)\;.
\end{align*}
\end{theorem}

\begin{corollary}\label{cor-Hopfalg}
Let $\k$ be a field, and for $i=1,2$, let 
$A_i^*\in\P_\A$ be a $n_i$-graded Hopf algebra functor. The maps:
\begin{align*}
&\Ext^*(A_1,A_2)\otimes \Ext^*(A_1,A_2)\to \Ext^*(A_1,A_2\otimes A_2)\to
\Ext^*(A_1,A_2)\\& 
\k\to \Ext^*(A_1,\k)\to \Ext^*(A_1,A_2)
 \\
&\Ext^*(A_1,A_2)\to \Ext^*(A_1,A_2\otimes A_2)\to \Ext^*(A_1,A_2)\otimes
\Ext^*(A_1,A_2)\\ 
&
\Ext^*(A_1,A_2)\to \Ext^*(A_1,\k)\to \k
\end{align*}
make $\Ext^*_{\P_\A}(A_1^*,A_2^*)$ into a $(1+n_1+n_2)$-graded Hopf
algebra.
\end{corollary}
\begin{proof}
To get corollary \ref{cor-Hopfalg} from theorem \ref{thm-Hopfmon}, we just
apply a graded version of lemma \ref{lm-Hopf-monoidal-funct-monoids}
 (which holds for additive functors).
\end{proof}

\begin{remark}
Corollary \ref{cor-Hopfalg} is a generalization of \cite[Lemma 1.11]{FFSS}.
Indeed, we don't require our functors $A_i^*$ to be exponential functors. In
section \ref{sec-applic}, we apply this corollary to Hopf algebra functors
which are \emph{not} exponential functors.
\end{remark}

\begin{remark}
In this section, the proofs rely on (1) Yoneda
isomorphisms for standard injectives, (2) adjunction between the sum and the
diagonal functors, (3) K\"unneth formulas. Properties (1) and (2) hold in the
category $\F_\A$ of ordinary functors with source an additive category $\A$ and
target $\kvect$. 
The K\"unneth formula also holds if one
assumes furthermore some finiteness conditions on the functors (either if
$F_1,F_2$ have finite dimensional values and
$G_1,G_2$ have injective resolutions by \emph{finite} sums of standard
injectives, or if $F_1,F_2$ have projective resolutions by \emph{finite}
sums of
projectives). Up to these slight finiteness conditions, the results of this
section holds in the category $\F_\A$.
This gives interesting applications for the stable cohomology of
the finite classical groups $O_{n,n}(\mathbb{F}_q)$ and 
$Sp_{n}(\mathbb{F}_q)$ with twisted coefficients \cite{DjamentVespa}.
\end{remark}

\section{Applications}\label{sec-applic}

\subsection{Stable products and coproducts for classical groups}

\begin{theorem}\label{thm-hopfmonstr-gp}
Let $\k$ be a field. 
Let $G_n$ be a product of copies of the groups $GL_n, Sp_n$ or $O_{n,n}$, and
let $F_1,F_2$ be strict polynomial functors adapted to $G_n$
(cf. terminology \ref{terminology}). If $O_{n,n}$ is a factor in $G_n$,
assume that $\k$ has odd characteristic. The stable rational
cohomology of $G_n$ is equipped with a coproduct:
$$\H^*(G_\infty, (F_1\otimes F_2)_\infty)\to \H^*(G_\infty, F_{1\;\infty})
\otimes \H^*(G_\infty, F_{2\;\infty})\;. $$
Together with the
usual cup product (cf. \S \ref{subsec-ratcohom}), 
they endow $\H^*(G_\infty,-)$ with the structure of a
graded Hopf monoidal functor.

Moreover, the cup product is a section of the coproduct. 
\end{theorem}
\begin{proof}
We consider the usual graded Hopf algebra structure $\Gamma^*(V)$, with
$\Gamma^d(V)$ in degree $2d$, cf. paragraph \ref{subsec-nota}.
Let $F_G\in\P_G$ be the characteristic functor associated to $G_n$. If we set
$V=F_G$, the divided powers of $F_G$ are a graded Hopf
algebra, or equivalently a graded Hopf monoidal functor. To be more specific, 
the product $\mu$ and the coproduct $\lambda$ are given by the formulas:
\begin{align*}\mu: 
 \Gamma^*(F_G(V&))\otimes \Gamma^*(F_G(W))\to \Gamma^*(F_G(V\oplus
 W))^{\otimes 2}\\& \simeq  \Gamma^*(F_G(V\oplus
 W)^{\oplus 2})\to \Gamma^*(F_G(V\oplus
 W))\;,\\
\lambda: 
\Gamma^*(F_G(V&\oplus
 W))\to 
\Gamma^*(F_G(V\oplus
 W)^{\oplus 2})\\&\simeq
\Gamma^*(F_G(V\oplus
 W))^{\otimes 2}\to \Gamma^*(F_G(V))\otimes \Gamma^*(F_G(W))
\;.
\end{align*} 
In particular, one checks that $\lambda\circ\mu=\Id$. 
Thus, by theorem \ref{thm-Hopfmon}, $\Ext^*(\Gamma^\star(F_G),-)$ is a
bigraded Hopf
monoidal functor, and $\lambda\circ\mu=\Id$ implies that the external cup
product is a section of the coproduct. 

Since the divided powers of $F_G$ are concentrated in even
degree, we may forget the grading arising from the divided powers (cf. remark
\ref{rk-degrees}) and $\Ext^*(\Gamma^\star(F_G),-)$ is a
$*$-graded Hopf
monoidal functor. Then it suffices to apply theorem \ref{thm-generique} to
conclude the proof.
\end{proof}

\begin{corollary}\label{cor-inj}
Let $\k$ be a field. 
Let $G_n$ be a product of copies of the groups $GL_n, Sp_n$ or $O_{n,n}$, and
let $F_1,F_2$ be two functors of degree $d_1,d_2$ adapted to $G_n$. 
If $O_{n,n}$ is a factor in $G_n$,
assume that $\k$ has odd characteristic.
For all $n$ such that $2n\ge d_1+ d_2$, the cup product
induces a injection:
$$ \H^*(G_n, (F_1)_n)
\otimes \H^*(G_n, (F_2)_n)\hookrightarrow
\H^*(G_n, (F_1)_n\otimes (F_2)_n)\;.  $$ 
\end{corollary}

\begin{remark}
The injectivity in odd degree cohomological degree 
does not contradict the usual commutativity formula $x\cup
y=(-1)^{\deg(x)\deg(y)} y\cup x$. Indeed, this latter formula holds only for 
\emph{internal} cup products. If $\tau$ denotes the
isomorphism $(F_1)_n\otimes (F_2)_n\simeq (F_2)_n\otimes (F_1)_n$, the
commutativity relation for external cup products is $x\cup
y=(-1)^{\deg(x)\deg(y)} \H^*(G_n,\tau)(y\cup x)\;.$ 
\end{remark}

\begin{corollary}\label{cor-Hopfalggp}
Let $\k$ be a field. 
Let $G_n$ be a product of copies of the groups $GL_n, Sp_n$ or $O_{n,n}$, and
let $A^*$ be an $n$-graded strict polynomial functor adapted to $G_n$, endowed
with the structure of a Hopf algebra. If $O_{n,n}$ is a factor in $G_n$,
assume that $\k$ has odd characteristic. The usual cup product
$\H^*(G_\infty,A^*_\infty)^{\otimes 2}\to \H^*(G_\infty,A^*_\infty)$ may be
supplemented with a coproduct 
$\H^*(G_\infty,A^*_\infty)\to\H^*(G_\infty,A^*_\infty)^{\otimes 2}$ which
endow $\H^*(G_\infty,A^*_\infty)$ with the structure of a $(1+n)$-graded Hopf
algebra. 
\end{corollary}

\subsection{A new statement for the universal classes}

As an application of theorem \ref{thm-hopfmonstr-gp}, we give a nicer
formulation  of the existence of the
universal cohomology classes \cite[Thm 4.1]{TVdK}. Consider the divided powers
$\Gamma^*(\H^*(GL_\infty, gl^{(1)}_\infty))$, with the usual Hopf algebra
structure but regraded in the following way: the bidegree
of an element in $\Gamma^{i}(\H^j(GL_\infty, gl^{(1)}_\infty))$ is $(2ij, 2i)$. We easily
get:
\begin{corollary}\label{cor-univ} Let $\k$
be a field of positive characteristic $p$. 
The existence of the universal cohomology classes 
is equivalent to the following statement.
  
There is a bigraded Hopf algebra morphism 
$$\psi:\Gamma^*(\H^*(GL_\infty, gl^{(1)}_\infty))\to 
\H^*(GL_\infty, \Gamma^*(gl^{(1)}_\infty))\;,$$
such that for all $n\ge p$ the following composite is a monomorphism:
$$\Gamma^*(\H^*(GL_\infty, gl^{(1)}_\infty))
\xrightarrow[]{\psi} \H^*(GL_\infty, \Gamma^*(gl^{(1)}_\infty))
\xrightarrow[]{\phi_{n,\infty}} \H^*(GL_n, \Gamma^*(gl_n^{(1)}))\;.$$
\end{corollary}

\subsection{Cohomology computations for the orthogonal and symplectic groups}
In \cite{DjamentVespa}, Djament and Vespa showed how to obtain cohomological
computations for the orthogonal and symplectic groups from
the computations of
\cite{FFSS}. Their method adapts easily to the strict polynomial
functor setting. For all $r\ge 0$, we denote by $I^{(r)}\in\P$ the $r$-th Frobenius twist
\cite[(v) p.224]{FS}. We consider the Hopf algebra $S^*(I^{(r)})$ (resp. 
$\Lambda^*(I^{(r)})$) with
$S^d(I^{(r)})$ in degree $2d$ (resp. with $\Lambda^d(I^{(r)})$ in degree $d$).
We have:

\begin{theorem}\label{thm-calcul}
Let $\k$ be a field of odd characteristic. Let $r$ be a nonnegative integer.

\begin{itemize}
\item[(i)]
The bigraded Hopf algebra 
$\H^*(O_{\infty,\infty}, S^\star(I^{(r)})_\infty)$ 
is a symmetric Hopf algebra on
generators $e_m$ of bidegree $(2m, 4)$ for $0\le m< p^r$.
\item[(ii)] 
The bigraded Hopf algebra
$\H^*(Sp_\infty, S^\star(I^{(r)})_\infty)$ is trivial.
\item[(iii)] The bigraded Hopf algebra
$\H^*(O_{\infty,\infty}, \Lambda^\star(I^{(r)})_\infty)$
is trivial.
\item[(iv)] The bigraded Hopf algebra
$\H^*(Sp_\infty, \Lambda^\star(I^{(r)})_\infty)$ is a divided power Hopf algebra on generators $e_m$ of bidegree $(2m,2)$ 
for $0\le m< p^r$. 
\end{itemize}
\end{theorem}

We get a proof by following closely 
\cite[Section 4]{DjamentVespa}. For sake
of completeness, we give some details in the remainder of this paragraph.
Let $\k$ be a field of characteristic $p>2$. 
As in section \ref{sec-3}, we denote by $\P(1,1)$ the category of strict
polynomial functors with source $\V^\op_\k\times \V_\k$ and by $\P$ the
category of strict polynomial functors with source $\V_\k$. We also denote by
$\P(2)$ the category strict polynomial functors with source 
$\V_\k\times \V_\k$.

Let $E^*=S^*(I^{(r)})$ (with $S^d(I^{(r)})$ in degree $2d$) 
or $E^*=\Lambda^*(I^{(r)})$ (with $\Lambda^d(I^{(r)})$ 
in degree $d$), or more generally let $E^*$ be a `skew-commutative 
Hopf exponential functor'
(cf. \cite[p. 675 and Def. 1.9]{FFSS}. Equivalently, $E^*$ is a graded
functor in $\P$ satisfying all the
hypotheses of lemma \ref{lm-ExpHopf}). 
We wish to compute the bigraded 
Hopf algebra $\Ext^*_\P(\Gamma^\star(F_G), E^*)$, with $F_G=S^2$ or $F_G=\Lambda^2$ (it
is a trigraded Hopf algebra by corollary \ref{cor-Hopfalg}, and 
we may drop the gradation on
$\Gamma^\star(F)$, since this gradation is concentrated in even degrees, cf.
remark \ref{rk-degrees}). Indeed, 
by theorem \ref{thm-generique}, this bigraded Hopf
algebra is isomorphic to the bigraded Hopf algebra
$\H^*(G_\infty,
E^*_\infty)$ with $G_n=O_{n,n}$ for $F_G=S^2$, and $G_n=Sp_n$ for
$F_G=\Lambda^2$.

To do the computation, it suffices to compute the bigraded Hopf algebra
$\Ext^*_\P(\Gamma^\star(\otimes^2), E^*)$, together with
the involution $\theta$ 
of bigraded Hopf
algebras, induced by the
permutation $V\otimes V\simeq V\otimes V$ which exchanges the factors of
$\otimes^2$. Indeed, since $\k$ has characteristic $p\ne 2$, $F_G$ is a direct
summand in $\otimes^2$. As a result, $\H^*(G_\infty,
E^*_\infty)$ equals the image of $(1+\theta)/2$ 
in the orthogonal case and of $(1-\theta)/2$ in the symplectic case. So we now
concentrate on $\Ext^*_\P(\Gamma^\star(\otimes^2), E^*)$. 

Let $I\in\P$ denote the identity functor of $\V_\k$. 
Using the sum diagonal adjunction of remark \ref{rk-2adj}, 
and the exponential isomorphism for $E^*$,
we obtain an isomorphism of multigraded vector spaces (on the right, we 
take the total degree of $E^*\boxtimes E^*$):
\begin{align*}\Ext^*_P(\Gamma^\star(\otimes^2),E^*)\simeq 
\Ext^*_{P(2)}(\Gamma^\star(I\boxtimes I),E^*\boxtimes E^*)\;.
\end{align*}
Now if $B\in\P(2)$ is a strict polynomial functor with 2 covariant variables, one may precompose
the first variable of $B$ by the duality functor $-^\vee:\V_\k\to \V_\k^\op,
V\mapsto V^{\vee}$. One obtains a strict polynomial functor
$B(-^\vee,-)\in\P(1,1)$. This yields an equivalence of categories between
$\P(2)$ and $\P(1,1)$. Let $(E^i)^\sharp$ denote the strict polynomial 
functor $V\mapsto E^i(V^\vee)^\vee$. Using \cite[Thm 1.5 (1.11)]{FF} we obtain
isomorphisms of bigraded vector spaces (recall that we don't take
the gradation of $\Gamma^\star(\otimes^2)$ into account):
\begin{align*} 
\nonumber\Ext^*_{P(2)}(\Gamma^\star(I\boxtimes I),E^*\boxtimes E^*)&\simeq 
\Ext^*_{P(1,1)}(\Gamma^\star(gl),E^*(-^\vee)\boxtimes E^*)\\&\simeq
\Ext^*_\P(E^{*\;\sharp}, E^*)\;.
\end{align*}
To sum up, we have an isomorphism of bigraded vector spaces (on the left we
don't take the gradation of $\Gamma^*(\otimes^2)$ into account, on the right
we take the total gradation associated to the gradations of $E^{*\;\sharp}$
and $E^*$):
\begin{align*}\Ext^*_P(\Gamma^\star(\otimes^2),E^*)\simeq 
\Ext^*_\P(E^{*\;\sharp}, E^*)\quad (*)
\end{align*}
Both objects have a bigraded Hopf algebra structure by corollary
\ref{cor-Hopfalg}. We need the Hopf algebra structure of
$\Gamma^\star(\otimes^2)$ to define the bigraded Hopf algebra structure on the
left but not to define the one on the right. Nonetheless:
\begin{lemma}[{\cite[Prop 4.10]{DjamentVespa}}]\label{lm-dv1} 
For all `skew commutative
exponential functor' $E^*\in\P$, 
the isomorphism $(*)$ is compatible with the bigraded Hopf algebra structures.
\end{lemma}

For $E^*=S^*(I^{(r)})$ or $E^*=\Lambda^*(I^{(r)})$, the Hopf algebra 
$\Ext^*_\P(E^{*\;\sharp}, E^*)$ is
computed in \cite[Thm 5.8]{FFSS}. So it remains to describe how 
the involution $\theta$ acts on these extension groups.
For all $F,G$, we have an isomorphism (see for example \cite[lemma
1.12]{FFSS}) $\Ext^*_\P(F^\sharp,G)\simeq
\Ext^*_\P(G^\sharp,F)$. With $F=G=E^*$, we obtain an involution:
$$\widetilde{\theta}:\Ext^*_\P(E^{*\;\sharp}, E^*)\xrightarrow[]{\simeq}
\Ext^*_\P(E^{*\;\sharp}, E^*)\;.$$
\begin{lemma}[{\cite[Lemme 4.12]{DjamentVespa}}]\label{lm-dv2} 
For all `skew commutative Hopf
exponential functor' $E^*\in\P$, we denote by $\widetilde{\theta}^*$ the
involution of $\Ext^*_\P(E^{*\;\sharp}, E^*)$ whose restriction to
$\Ext_\P^*(E^{i\,\sharp},E^j)$ equals $(-1)^{ij}\widetilde{\theta}$.
We have a commutative diagram:
$$\xymatrix{
\Ext^*_P(\Gamma^\star(\otimes^2),E^*)\ar[r]^-{\simeq}_{(*)}\ar[d]^-{\theta}
&\Ext^*_\P(E^{*\;\sharp}, E^*)\ar[d]^-{\widetilde{\theta}^*}\\
\Ext^*_P(\Gamma^\star(\otimes^2),E^*)\ar[r]^-{\simeq}_{(*)}
&\Ext^*_\P(E^{*\;\sharp}, E^*)\;.
}$$
\end{lemma}

We are now ready to use the computations of \cite{FFSS}. We first recall
the results we need. If $V^*$ is a graded vector space concentrated in even
degrees, we consider the
vector spaces $S^*(V^*)$ bigraded in the following way: the bidegree of an
element $S^i(V^j)$ is $(ij, 4i)$. With this grading, 
the usual Hopf algebra structure on the
symmetric powers makes $S^*(V^*)$ into a bigraded Hopf algebra. We have (recall
that an element in $\Ext^k_\P(\Gamma^\ell(I^{(r)}), S^m(I^{(r)}))$ has
bidegree $(k, 2\ell+2m)$):
\begin{lemma}[{\cite[Thm 4.5 and Thm 5.8]{FFSS}}]\label{lm-calcul-FFSS1}
For all $n\ge 0$, the bigraded Hopf algebra multiplication 
$$\Ext^*_\P(\Gamma^1(I^{(r)}),S^1(I^{(r)})))^{\otimes n}\to
\Ext^*_\P(\Gamma^n(I^{(r)}),S^n(I^{(r)})))$$
is surjective. It induces an isomorphism of Hopf algebras: 
$$S^*(\Ext^*_\P(I^{(r)},I^{(r)}))\simeq
\Ext^*_\P(\Gamma^*(I^{(r)}),S^*(I^{(r)})))\;.$$ 
\end{lemma}
Since the involution $\theta$ is compatible with the Hopf algebra structure,
the first part of lemma \ref{lm-calcul-FFSS1} shows that knowing
the involution $\widetilde{\theta}$ on $\Ext^*_\P(I^{(r)},I^{(r)})$is
sufficient to determine $\theta$.  The involution
$\widetilde{\theta}$ is already computed by Djament and Vespa:
\begin{lemma}[{\cite[Lemme 4.13]{DjamentVespa}}]\label{lm-dv3} 
The involution $\widetilde{\theta}$ equals the identity map.
\end{lemma}

Thus, by lemmas \ref{lm-dv2}, \ref{lm-calcul-FFSS1} and \ref{lm-dv3},
the map $(1+\theta)/2: \Ext^*_\P(\Gamma^\star(\otimes^2),S^*(I^{(r)}))\to
\Ext^*_\P(\Gamma^\star(\otimes^2),S^*(I^{(r)}))$ equals the identity map, so that we
have:
\begin{align*}&\Ext^*_\P(\Gamma^\star(S^2),S^*(I^{(r)}))\simeq 
\Ext^*_\P(\Gamma^\star(\otimes^2),S^*(I^{(r)}))\;,\\
&\Ext^*_\P(\Gamma^\star(\Lambda^2),S^*(I^{(r)}))\simeq 
0\;.
\end{align*}
Now we may apply lemmas \ref{lm-dv1} and \ref{lm-calcul-FFSS1} to conclude the
proof of \ref{thm-calcul}(i) and (ii). 
The computation for the
exterior powers $\Lambda^*(I^{(r)})$ is similar.


\begin{thebibliography}{99}
%\bibitem{Chalupnik} Cha{\l}upnik, Marcin Koszul duality and 
%extensions of exponential functors. Adv. Math.  218  (2008),  
%no. 3, 969--982.
\bibitem{AJ} H.H.~Andersen, J.C.~Jantzen, Cohomology of induced 
representations for algebraic groups.  Math. Ann.  269  (1984),  
no. 4, 487--525.
\bibitem{Buehler} T.~B\"uhler, Exact Categories, Expositiones Mathematicae 28 (2010) 1--69,
doi: 10.1016/j.exmath.2009.04.004
\bibitem{Chal1} M.~Cha{\l}upnik, Extensions of strict polynomial functors. 
Ann. Sci. \'Ecole Norm. Sup. (4)  38  (2005),  no. 5, 773--792.
\bibitem{Chaput} E.~Chaput, M.~Romagny, On the adjoint quotient of Chevalley
groups over arbitrary base schemes, {\tt ArXiv:0805.2140}
\bibitem{CPSVdK} E. Cline, B. Parshall, L. Scott, W. van der Kallen,
Rational and generic cohomology, 
Invent. Math. 39 (1977), 143-163.
\bibitem{DCP} C.~de Concini, C.~Procesi,  A characteristic free approach to 
invariant theory.  Advances in Math.  21  (1976), no. 3, 330--354.
\bibitem{DjamentVespa} A.~Djament, C.~Vespa, Sur l'homologie des groupes
orthogonaux et symplectiques \`{a} coefficients tordus,  Ann. Sci. \'Ecole Norm. Sup. (4)  43  (2010),  no. 3.
%{\tt arXiv:0808.4035}. 
\bibitem{FF}V. Franjou, E. M. Friedlander, Cohomology of bifunctors,
Proceedings of the London Mathematical Society (2008), doi:10.1112/plms/pdn005
\bibitem{FFSS} V.~Franjou, E.~Friedlander, A.~Scorichenko, A.~Suslin, 
General linear and functor cohomology over finite fields,   Ann. of Math. (2)
  150  (1999),  no. 2, 663--728.
\bibitem{FS} E.~Friedlander, A.~Suslin, Cohomology of finite group schemes over a field,
Invent. Math. 127 (1997), 209--270.
\bibitem{Grot} A.~Grothendieck, Sur quelques points d'alg\`ebre homologique. 
(French)  T\^ohoku Math. J. (2)  9  1957 119--221.
\bibitem{Jantzen}J.-C.\ Jantzen, Representations of Algebraic Groups,
Mathematical Surveys and Monographs vol. 107, Amer.\ Math.\ Soc.,
Providence, 2003.
\bibitem{BKINV} M.-A.~Knus, A.~Merkurjev, M.~Rost, J.-P.~Tignol, The book of involutions. American Mathematical Society Colloquium Publications, 44. American Mathematical Society, Providence, RI, 1998. xxii+593 pp. ISBN: 0-8218-0904-0.
\bibitem{ML} S.~Mac~Lane, Homology. Reprint of the 1975 edition. 
Classics in Mathematics. Springer-Verlag, Berlin, 1995. x+422 pp. 
ISBN: 3-540-58662-8.
\bibitem{MLCat} S.~Mac~Lane, Categories for the working mathematician. Second edition. 
Graduate Texts in Mathematics, 5. Springer-Verlag, New York, 
1998. xii+314 pp. ISBN: 0-387-98403-8.
\bibitem{Pira} T.~Pirashvili, Introduction to functor homology.  
Rational representations, the Steenrod algebra and functor homology,  
1--26, Panor. Synth\`eses, 16, Soc. Math. France, Paris, 2003.
%\bibitem{Tanre} D.~Tanr\'e, Homotopie rationnelle: mod\`eles de Chen, 
%Quillen, Sullivan.  
%(French) [Rational homotopy: Chen, Quillen and Sullivan models]
%Lecture Notes in Mathematics, 1025. Springer-Verlag, Berlin, 1983. 
%x+211 pp. ISBN: 3-540-12726-7.
\bibitem{Quillen} D.~Quillen, Higher algebraic $K$-theory. I.  Algebraic
$K$-theory, I: Higher $K$-theories (Proc. Conf., Battelle Memorial Inst.,
Seattle, Wash., 1972),  pp. 85--147. Lecture Notes in Math., Vol. 341,
Springer, Berlin 1973. 
\bibitem{SFB} A.~Suslin, E.~Friedlander, C.~Bendel, 
 Infinitesimal $1$-parameter subgroups and cohomology.  J. Amer. Math. Soc.  
 10  (1997),  no. 3, 693--728.
\bibitem{Touze}A. Touz\'e, Universal classes for algebraic groups, 
Duke Math. J.  151  (2010),  no. 2, 219--249.
%{\tt arXiv:0809.0989}.
\bibitem{TVdK} A. Touz\'e, W.~van der Kallen, Bifunctor cohomology and
cohomological finite generation for reductive groups, 
Duke Math. J.  151  (2010),  no. 2, 251--278.
%{\tt arXiv:0809.1014}.
%\bibitem{VdK} W.~van der Kallen, Cohomology with Grosshans graded coefficients, In: Invariant Theory in All Characteristics, Edited by: H. E. A. Eddy Campbell and David L. Wehlau, CRM Proceedings and Lecture Notes, Volume 35 (2004) 127-138, Amer. Math. Soc., Providence, RI, 2004.
\end{thebibliography}
\end{document}